\documentclass[12pt]{amsart}
\usepackage{amsmath, amsthm, amscd, amsfonts}
\usepackage[all]{xy}

\newtheorem{theorem}{Theorem}[section]
\newtheorem{lemma}[theorem]{Lemma}

\newtheorem{cor}[theorem]{cor}
\newtheorem{prop}[theorem]{Proposition}
\theoremstyle{definition}
\newtheorem{definition}[theorem]{Definition}
\theoremstyle{definition}
\newtheorem{remark}[theorem]{Remark}
\theoremstyle{definition}

\theoremstyle{definition}
\newtheorem{example}[theorem]{Example}

\usepackage{mathrsfs} 
\usepackage{blindtext}
\usepackage{multicol}
\usepackage{dsfont}
\usepackage{mathtools}
\usepackage{verbatim}
\usepackage{xcolor}
\usepackage{amsthm}
\usepackage{fancyhdr}
\usepackage{lmodern}
\usepackage{amssymb}
\usepackage{latexsym}
\usepackage{graphicx}
\usepackage{etoolbox} 
\usepackage{amsmath}
\usepackage{multicol}
\begin{document}

	%\title[Short title]{Title of The Paper}

	\title[vN Reg. $\mathcal{C}^{\infty}-$Rings, Boolean alg.]{Classification of Boolean Algebras through  von Neumann regular $\mathcal{C}^{\infty}-$Rings}

	%\author[F. Author, S. Author]{First Author$^1$ and Second Author$^2$}
	
	\author[J. Berni, H. Mariano]{Jean Cerqueira Berni$^1$ and Hugo Luiz Mariano$^2$}

	\address{$^{1}$ Department of Mathematics, S\~{a}o Paulo State University - UNESP, 13506-900 , S\~{a}o Paulo, Brazil.}
	
	%\email{first@xx.yy}
	\email{j.berni@unesp.br}
	
	\address{$^{2}$ Department of Mathematics, University of S\~{a}o Paulo, 05508-090, S\~{a}o Paulo, Brazil.}
	
	%\email{second@xx.yy}
	\email{hugomar@ime.usp.br}
	
	
	%\keywords{At least three words separated with comma, and at most 10 words.}
	

	\begin{abstract}
		{In this paper, we introduce the concept of a "von Neumann regular $\mathcal{C}^{\infty}$-ring", which is a model for a specific equational theory. We delve into the characteristics of these rings and demonstrate that each Boolean space can be effectively represented as the image of a von Neumann regular $\mathcal{C}^{\infty}$-ring through a specific functor. Additionally, we establish that every homomorphism between Boolean algebras can be expressed through a $\mathcal{C}^{\infty}$-ring homomorphism between von Neumann regular $\mathcal{C}^{\infty}$-rings.
		}
		%{In this paper we propose a notion of ``von Neumann regular $\mathcal{C}^{\infty}-$ring'' - a model of a certain equational theory -  and describe some of their properties. %We prove, using two different, yet equally important methods, that the category of von Neumann regular $\mathcal{C}^{\infty}-$rings is a reflective  subcategory of $\mathcal{C}^{\infty}{\rm \bf Ring}$, the category of all $\mathcal{C}^{\infty}-$rings and their homomorphisms. 
			%We prove that every Boolean space can be realised as the image of a von Neumann regular $\mathcal{C}^{\infty}-$ring under a particular functor, and that every  homomorphism between Boolean algebras can be represented by a $\mathcal{C}^{\infty}-$ring homomorphism between von Neumann regular $\mathcal{C}^{\infty}-$rings.}
	\end{abstract}
	
	\maketitle

	\section{Introduction}
	
	As per I. Moerdijk and G. Reyes, $\mathcal{C}^{\infty}$-rings were primarily introduced for their applications in Singularity Theory, as  Weil algebras and jet spaces, and for constructing topos-models for Synthetic Differential Geometry. A remarkable topos is the one of sheaves over the small site comprising the category of all germ-determined $\mathcal{C}^{\infty}$-rings, along with an appropriate Grothendieck topology (for detailed information, refer to \cite{moerdijk1991models}).
	
	In recent years, the authors have been focused  on exploring $\mathcal{C}^{\infty}$-rings and their applications, cf. \cite{berni2018thesis}, \cite{berni2018classifying}, \cite{berni2019topics}, \cite{berni2019neumann}, \cite{berni2020boletin}, \cite{berni2022universitaria}, \cite{berni2022universal}, \cite{berni2022order} and \cite{berni2023separation}. In this study, we present a valuable concept of a von Neumann regular $\mathcal{C}^{\infty}$-ring, showcasing pertinent results and detailing both their properties and those of the category they form.

	As an application, we establish that every Boolean algebra can be canonically represented through a von Neumann regular $\mathcal{C}^{\infty}$-ring. Additionally, we demonstrate that each homomorphism between Boolean algebras can be depicted by a specific $\mathcal{C}^{\infty}$-ring homomorphism between von Neumann regular $\mathcal{C}^{\infty}$-rings. Specifically, the investigation of von Neumann regular $\mathcal{C}^{\infty}$-rings extends to the classification of Boolean spaces (and/or Boolean algebras) and their morphisms. A critical tool in developing these findings is the alignment between the concepts of Boolean spaces and "profinite spaces." Profinite spaces, along with their morphisms, constitute a category denoted as ${\rm \bf Profinite}$, forming the basis for the notion of a "condensed set." A condensed set is characterized as a functor (a specific sheaf) denoted by $C: {\rm \bf Profinite}^{\rm op} \to {\rm \bf Set}$. These condensed sets, in turn, represent a fundamental concept that plays a pivotal role in the recent theory of "Condensed Mathematics" developed by Dustin Clausen and Peter Scholze.
	
	%As an application, we prove that every Boolean algebra can be canonically represented by means of a von Neumann regular $\mathcal{C}^{\infty}-$ring and that every homomorphism between Boolean algebras can be represented by a certain $\mathcal{C}^{\infty}-$ring homomorphism between von Neumann regular $\mathcal{C}^{\infty}-$rings. In particular, the study of von Neumann regular $\mathcal{C}^{\infty}-$rings applies to the classification of Boolean spaces (and/or Boolean algebras) and their morphisms: a relevant tool for developing these results is the coincidence between the notion of  Boolean spaces and  ``profinite spaces'': the profinite spaces and their morphisms form a category, denoted as ${\rm \bf Profinite}$, serving as the foundation for the concept of a "condensed set." A condensed set is characterized as a functor (a specific sheaf) denoted by $C: {\rm \bf Profinite}^{\rm op} \to {\rm \bf Set}$. The condensed sets, in turn, represent a fundamental concept that is crucial for the development of the recent theory of "Condensed Mathematics" by Dustin Clausen and Peter Scholze.

	% the profinite spaces and their morphisms compose a category, ${\rm \bf Profinite}$, and constitute the basis of the notion of a ``condensed set'', which is a functor (a certain sheaf) $C : {\rm \bf Profinite}^{\rm op} \to {\rm \bf Set}$; the condensed sets, in  turn, is the fundamental notion  necessary to  Dustin Clausen and Peter Scholze to develop the recent theory of ``Condensed Mathematics'' (\cite{Scholze2019LecturesOC}).

	\subsection*{Overview of the Paper:} In {\bf Section 2}, we outline the fundamental preliminary concepts of $\mathcal{C}^{\infty}$-rings, delving into their universal algebra and their 'smooth commutative algebra.'
	
	Moving on to {\bf Section 3}, we introduce the concept of a von Neumann regular $\mathcal{C}^{\infty}$-ring, essentially a $\mathcal{C}^{\infty}$-ring with an underlying commutative unital ring being a von Neumann regular ring. We then investigate the main properties and characterizations associated with this concept.
	
	%{\bf Section 2} provides the basic preliminary notions of   $\mathcal{C}^{\infty}-$rings concerning their universal algebra  and their ``smooth commutative algebra''.
	%In {\bf Section 3}, we present a notion of von Neumann regular $\mathcal{C}^{\infty}-$ring - which is basically a $\mathcal{C}^{\infty}-$ring whose underlying commutative unital ring is a von Neumann regular ring - and explore the main properties and characterizations of this concept. 
	%In {\bf Section 3} we show, using two different methods, that every $\mathcal{C}^{\infty}-$ring has a ``closest'' von Neumann regular $\mathcal{C}^{\infty}-$ring, and define their von Neumann regular hull; moreover we describe some properties of the reflection functor as to the preservation of finite products. 
	
	In \textbf{Section 4}, we explore their practical applications. Firstly, we establish that each continuous map between profinite spaces can be accurately portrayed by a continuous map between spectral spaces, which are canonically associated with a $\mathcal{C}^{\infty}$-ring homomorphism between von Neumann regular $\mathcal{C}^{\infty}$-rings (see \textbf{Theorem \ref{boospvn}}). Subsequently, we deduce that every homomorphism between Boolean algebras can be distinctly represented through a $\mathcal{C}^{\infty}$-ring homomorphism between von Neumann regular $\mathcal{C}^{\infty}$-rings (refer to \textbf{Theorem \ref{qmq}}).

	\section{Preliminaries on $\mathcal{C}^{\infty}$-rings}
	
	In this section we provide the main preliminary notions on   $\mathcal{C}^{\infty}-$rings, with respect to their universal algebra (cf. \cite{berni2022universal}) and to their ``smooth commutative algebra" (cf. \cite{berni2019neumann}).
	
	In order to formulate and study the concept of $\mathcal{C}^{\infty}-$ring, we use a first order language, $\mathcal{L}$, with a denumerable set of variables (${\rm Var}(\mathcal{L}) = \{ x_1, x_2, \cdots, x_n, \cdots\}$), whose nonlogical symbols are the symbols of $\mathcal{C}^{\infty}-$\-functions from $\mathbb{R}^m$ to $\mathbb{R}^n$, with $m,n \in \mathbb{N}$, \textit{i.e.}, the non-logical symbols consist only of function symbols: for each $n \in \mathbb{N}$, the $n-$ary \textbf{function symbols} of the set $\mathcal{C}^{\infty}(\mathbb{R}^n, \mathbb{R})$, \textit{i.e.}, $\mathcal{F}_{(n)} = \{ f^{(n)} \mid f \in \mathcal{C}^{\infty}(\mathbb{R}^n, \mathbb{R})\}$. Thus, the set of function symbols of our language is given by:
	$$\mathcal{F} = \bigcup_{n \in \mathbb{N}} \mathcal{F}_{(n)} = \bigcup_{n \in \mathbb{N}} \mathcal{C}^{\infty}(\mathbb{R}^n).$$
	Note that our set of constants is identified with the set of all $0-$ary function symbols, \textit{i.e.}, $\mathcal{C} = \mathcal{F}_{(0)} = \mathcal{C}^{\infty}(\mathbb{R}^0) \cong \mathcal{C}^{\infty}(\{ *\})$.
	
	The terms of this language are defined, in the usual way, as the smallest set which comprises the individual variables, constant symbols and $n-$ary function symbols followed by $n$ terms ($n \in \mathbb{N}$).
	
	Functorially, a (set-theoretic) $\mathcal{C}^{\infty}-$ring is a finite product preserving functor from the category $\mathcal{C}^{\infty}$, whose objects are of the form $\mathbb{R}^n$, $n \in \mathbb{N}$, and whose morphisms are the smooth functions between them, \textit{i.e.}, a finite product preserving functor:
	$$A: \mathcal{C}^{\infty} \rightarrow {\rm \bf Set}$$
	
	Apart from the functorial definition and the ``first-order language'' definition we just gave, there are many equivalent descriptions. We focus, first, on the universal-algebraic description of a $\mathcal{C}^{\infty}-$ring in ${\rm \bf Set}$, given in the following:
	
	\begin{definition}\label{cabala} A \textbf{$\mathcal{C}^{\infty}-$structure}  is a pair $ \mathfrak{A} =(A,\Phi)$, in which $A$ is a non-empty set and:
		
		$$\begin{array}{cccc}
			\Phi: & \bigcup_{n \in \mathbb{N}} \mathcal{C}^{\infty}(\mathbb{R}^n, \mathbb{R})& \rightarrow & \bigcup_{n \in \mathbb{N}} {\rm Func}\,(A^n; A)\\
			& (f: \mathbb{R}^n \stackrel{\mathcal{C}^{\infty}}{\to} \mathbb{R}) & \mapsto & \Phi(f) := (f^{A}: A^n \to A)
		\end{array},$$
		
		\noindent is a function, that is, $\Phi$ interprets the \textbf{symbols}\footnote{here considered simply as syntactic symbols rather than functions.} of all smooth real functions of $n$ variables as $n-$ary function symbols on $A$.
	\end{definition}
	
	We call a $\mathcal{C}^{\infty}-$struture, $\mathfrak{A} = (A, \Phi)$, a \textbf{$\mathcal{C}^{\infty}-$ring}, whenever it preserves projections and all equations between smooth functions. More precisely, we have the following:
	
	\begin{definition}\label{CravoeCanela}Let $\mathfrak{A}=(A,\Phi)$ be a $\mathcal{C}^{\infty}-$structure. We say that $\mathfrak{A}$ (or, when there is no danger of confusion, $A$) is a \textbf{$\mathcal{C}^{\infty}-$ring} if the following is true:
		\begin{itemize}
			\item{ Given any $n,k \in \mathbb{N}$ and any projection $\pi_k: \mathbb{R}^n \to \mathbb{R}$, we have:
				
				$$(\forall x_1)\cdots (\forall x_n)(\Phi(\pi_k)(x_1, \cdots, x_n)=x_k).$$} 
			\item{For every $f, g_1, \cdots g_n \in \mathcal{C}^{\infty}(\mathbb{R}^m, \mathbb{R})$ with $m,n \in \mathbb{N}$, and every $h \in \mathcal{C}^{\infty}(\mathbb{R}^n, \mathbb{R})$ such that $f = h \circ (g_1, \cdots, g_n)$, one has:
				$$\Phi(f)(\vec{x})=\Phi(h)(\Phi(g_1)(\vec{x}), \cdots, \Phi(g_n)(\vec{x}))$$}
		\end{itemize}
		
	\end{definition}
	
	\begin{definition}Let $(A, \Phi)$ and $(B,\Psi)$ be two $\mathcal{C}^{\infty}-$rings. A function $\varphi: A \to B$ is called a \textbf{morphism of $\mathcal{C}^{\infty}-$rings} or \textbf{$\mathcal{C}^{\infty}$-homomor\-phism} if for any $n \in \mathbb{N}$ and any $f: \mathbb{R}^n \stackrel{\mathcal{C}^{\infty}}{\to} \mathbb{R}$, one has $\Psi(f)\circ \varphi^{(n)} = \varphi \circ \Phi(f)$, where $\varphi^{(n)} = (\varphi, \cdots, \varphi): A^n \to B^n$.
	\end{definition}

	\begin{remark} {\bf (on universal algebraic constructions)}
		It is not difficult to see that $\mathcal{C}^{\infty}-$structures, together with their morphisms (which we call $\mathcal{C}^{\infty}-$morphisms) compose a category, that we denote by $\mathcal{C}^{\infty}{\rm \bf Str}$, and that $\mathcal{C}^{\infty}-$rings, together with all the $\mathcal{C}^{\infty}-$morphisms between $\mathcal{C}^{\infty}-$\-rings (which we call $\mathcal{C}^{\infty}-$homomorphisms) compose a full subcategory, $\mathcal{C}^{\infty}{\rm \bf Ring}$. In particular, since $\mathcal{C}^{\infty}{\rm \bf Ring}$ is a ``variety of algebras'', \emph{i.e.} it is a class of $\mathcal{C}^{\infty}-$structures which satisfies a given set of equations, (or equivalently,  by \textbf{Birkhoff's HSP Theorem}) it is closed under substructures, homomorphic images and products. Moreover:\\
		$\bullet$ $\mathcal{C}^{\infty}{\rm \bf Ring}$ is a concrete category and the forgetful functor, $U: \mathcal{C}^{\infty}{\rm \bf Ring} \to {\rm \bf Set}$,
		creates directed inductive colimits;\\
		$\bullet$ Each set $X$ freely generates a $\mathcal{C}^{\infty}$-ring, $\mathcal{C}^{\infty}(\mathbb{R}^X)$. In particular, the free $\mathcal{C}^{\infty}$-ring on $n$ generators is (isomorphic to) $\mathcal{C}^{\infty}(\mathbb{R}^n)$, $n \in \mathbb{N}$. Moreover, $\mathcal{C}^{\infty}(\mathbb{R}^X) \cong \varinjlim_{X' \subset_{\rm fin} X} \mathcal{C}^{\infty}(\mathbb{R}^{X'})$ ;\\
		$\bullet$ Every $\mathcal{C}^{\infty}-$ring is the homomorphic image of some free $\mathcal{C}^{\infty}-$ring determined by some set, being isomorphic to the quotient of a free $\mathcal{C}^{\infty}-$ring by some congruence;\\
		$\bullet$ The congruences of $\mathcal{C}^{\infty}-$rings are classified by their ``ring-theoretical'' ideals - which are the ideals of a $\mathcal{C}^{\infty}-$ring, or the ``$\mathcal{C}^{\infty}-$ideals'';\\
		$\bullet$ In $\mathcal{C}^{\infty}{\rm \bf Ring}$ one defines ``the $\mathcal{C}^{\infty}-$coproduct'' between two $\mathcal{C}^{\infty}-$rings $\mathfrak{A} = (A,\Phi)$ and $\mathfrak{B}=(B,\Psi)$, denoted by $A \otimes_{\infty} B$;\\
		$\bullet$ Using free $\mathcal{C}^{\infty}-$rings and the $\mathcal{C}^{\infty}-$coproduct, one gets the ``$\mathcal{C}^{\infty}-$ring of polynomials'' on any set $S$ of variables with coefficients in $A$, given by $A\{ x_s \mid s \in S\} = A\otimes_{\infty}\mathcal{C}^{\infty}(\mathbb{R}^S)$.
	\end{remark} 
	
	\begin{remark}({\bf on smooth commutative algebra}). Every $\mathcal{C}^{\infty}-$ring has an underlying commutative unital ring, so there is a naturally defined forgetful functor $\widetilde{U}: \mathcal{C}^{\infty}{\rm \bf Ring} \to {\rm \bf CRing}$. Using such forgetful functor, one defines a $\mathcal{C}^{\infty}-$field (resp. $\mathcal{C}^{\infty}-$domain, local $\mathcal{C}^{\infty}-$ring) as a $\mathcal{C}^{\infty}-$ring $\mathfrak{A} = (A,\Phi)$ such that $\widetilde{U}(\mathfrak{A})$ is a field (resp. domain, local ring), when regarded as a commutative unital ring;\\
		$\bullet$ In $\mathcal{C}^{\infty}{\rm \bf Ring}$ one has the $\mathcal{C}^{\infty}-$ring of fractions of a $\mathcal{C}^{\infty}-$ring $A$ with respect to any subset $S$ of $A$, denoted by $A \stackrel{\eta_S}{\rightarrow} A\{ S^{-1}\}$, in the same sense one has the ring of fractions with respect to a subset of a commutative unital ring, defined by the following two properties: (i) given any $a \in S$, $\eta_S(a) \in A^{\times}$ and (ii) given any $\mathcal{C}^{\infty}-$ring $B$ and any $\mathcal{C}^{\infty}-$homomorphism $f: A \to B$ such that $(\forall a \in S)(f(a)\in B^{\times})$ there is a unique $\mathcal{C}^{\infty}-$homo\-morphism $\widetilde{f}: A\{ S^{-1}\} \to B$ such that $\widetilde{f}\circ \eta_S = f$.\\
		$\bullet$ The $\mathcal{C}^{\infty}-$ring of fractions can be constructed using universal algebraic tools, and it is given by the quotient $A\{ S^{-1}\} \cong A\{ x_s \mid s \in S\}/{\langle \{x_s \cdot s - 1 \mid s \in S\}\rangle}$.\\
		$\bullet$ I. Moerdijk and G. Reyes introduced the notion of the $\mathcal{C}^{\infty}-$radical of an ideal $I$ of a $\mathcal{C}^{\infty}-$ring $\mathfrak{A} = (A,\Phi)$ (thus, a ring-theoretical ideal) as the set:
		$$\sqrt[\infty]{I} = \{ a \in A \mid ( A/I)\{ {a+I}^{-1}\} \cong \{ 0\} \}$$
		$\bullet$ The $\mathcal{C}^{\infty}-$spectrum of a $\mathcal{C}^{\infty}-$ring $A$ is the topological space whose underlying set is $X = \{ \mathfrak{p} \subseteq A \mid (\mathfrak{p}\,\, \text{is a prime ideal})\& (\sqrt[\infty]{\mathfrak{p}} = \mathfrak{p})\}$ and whose topology is generated by $\mathcal{B} =\{ D^{\infty}(a) \mid a \in A\}$, where $D^{\infty}(a) = \{ \mathfrak{p} \in X \mid a \notin \mathfrak{p} \}$. Moreover, $\mathcal{B}$ is closed under finite intersections and arbitrary reunions. We denote this topological space by ${\rm Spec}^{\infty}(A)$. Sometimes, when there is no danger of confusion, we write ${\rm Spec}^{\infty}(A)$ to denote the underlying set to this topological space, instead of $X$;\\
		$\bullet$ The $\mathcal{C}^{\infty}-$radical of a $\mathcal{C}^{\infty}-$ideal $I$ of a $\mathcal{C}^{\infty}-$ring $A$ is characterised by:
		$$\sqrt[\infty]{I} = \bigcap \{ \mathfrak{p} \in {\rm Spec}^{\infty}\,(A) \mid I \subseteq \mathfrak{p}\}$$
		$\bullet$ There is an \textit{ad hoc} definition of saturation for $\mathcal{C}^{\infty}-$rings, the smooth saturation of a subset $S$ of a $\mathcal{C}^{\infty}-$ring $A$, given by $S^{\infty-{\rm sat}} = \{ a \in A \mid \eta_S(a) \in A\{ S^{-1}\}^{\times}\}$. The smooth saturation is related to the $\mathcal{C}^{\infty}-$radical of an ideal $I \subset A$ by $\sqrt[\infty]{I} = \{ a \in A \mid I \cap \{ a\}^{\infty-{\rm sat}} \neq \varnothing\}$;\\
		$\bullet$ Along with the notion of a $\mathcal{C}^{\infty}-$radical ideal we have the concept of a reduced $\mathcal{C}^{\infty}-$ring, which is a $\mathcal{C}^{\infty}-$ring $\mathfrak{A}=(A,\Phi)$ such that $\sqrt[\infty]{(0_{A})} = (0_A)$.\\
		$\bullet$ A $\mathcal{C}^{\infty}-$ringed space is a pair $(X, \mathcal{O}_X)$, where $X$ is a topological space and $\mathcal{O}_X: {\rm Open}\,(X)^{\rm op} \to \mathcal{C}^{\infty}{\rm \bf Ring}$ is a sheaf. A morphism of $\mathcal{C}^{\infty}-$ringed spaces is a pair $(f,f^{\sharp}): (X,\mathcal{O}_X) \to (Y, \mathcal{O}_Y)$ of $\mathcal{C}^{\infty}-$ringed spaces is a continuous map $f: X \to Y$ and a morphism of sheaves $f^{\sharp}: f^{\dashv}[\mathcal{O}_Y] \to \mathcal{O}_X$, where $f^{\dashv}[\mathcal{O}_Y]$ is described in \textbf{Definition 4.5} of \cite{joyce2019algebraic}.
	\end{remark}
	
	\section{On von Neumann regular $\mathcal{C}^{\infty}-$rings}

	We begin by registering a fact that is valid for any $\mathcal{C}^{\infty}-$ring regarding to its idempotents and localizations:
	
	\begin{lemma}\label{destino}Let $A$ be \underline{any} $\mathcal{C}^{\infty}-$ring and $e \in A$ an idempotent element. There are unique isomorphisms:
		$$A\{ e^{-1}\} \cong A/(1-e) \cong A \cdot e := \{ a \cdot e \mid a \in A\}$$
	\end{lemma}
	\begin{proof}
		It is straightforward, A detailed proof is given in \textbf{Lemma 1} of \cite{berni2019neumann}.
	\end{proof}

	Next we give a precise definition of a von Neumann regular $\mathcal{C}^{\infty}-$ring. Loosely speaking, it is a $\mathcal{C}^{\infty}-$ring $(A,\Phi)$ such that $\widetilde{U}(A,\Phi)$ is a von Neumann regular commutative unital ring.

	\begin{definition}\label{arnaldo}Let $\mathfrak{A}=(A,\Phi)$ be a $\mathcal{C}^{\infty}-$ring. We say that $\mathfrak{A}$ is a \index{von Neumann regular $\mathcal{C}^{\infty}-$ring}\textbf{von Neumann regular $\mathcal{C}^{\infty}-$ring} if one (and thus all) of the following equivalent\footnote{See, \textbf{Proposition 3.9}, given later.},  conditions is satisfied:
		\begin{itemize}
			\item[(i)]{$(\forall a \in A)(\exists x \in A)(a = a^2x)$;}
			\item[(ii)]{Every principal ideal of $A$ is generated by an idempotent element, \textit{i.e.},
				$$(\forall a \in A)(\exists e \in A)(\exists y \in A)(\exists z \in A)((e^2=e)\& (ey=a) \& (az = e))$$}
			\item[(iii)]{$(\forall a \in A)(\exists ! b \in A)((a=a^2b)\& (b = b^2a))$}
		\end{itemize}
	\end{definition}
	
	\begin{example}Consider the set $\mathbb{R}^m = \mathcal{C}^{\infty}\,(\{*\})\times \cdots \times \mathcal{C}^{\infty}\,(\{*\})$, together with the function:
		
		$$\begin{array}{cccc}
			\Phi^{(m)}: & \bigcup_{n \in \mathbb{N}}\mathcal{C}^{\infty}(\mathbb{R}^n, \mathbb{R}) & \rightarrow & \bigcup_{n \in \mathbb{N}}{\rm Func}\,((\mathbb{R}^m)^n, (\mathbb{R}^m))\\
			& \mathbb{R}^n  \stackrel{f}{\rightarrow} \mathbb{R} & \mapsto & \mathbb{R}^m \times \cdots \times \mathbb{R}^m \stackrel{\Phi^{(m)}(f)}{\rightarrow} \mathbb{R}^m
		\end{array}$$
		
		\noindent with:
		
		$$\begin{array}{cccc}
			\Phi^{(m)}(f): & (\mathbb{R}^m)^n & \rightarrow & \mathbb{R}^m \\
			&((x_j^1)_{j=1}^{m}, \cdots, (x_j^n)_{j=1}^{m}) & \mapsto & (f((x_1^i)_{i=1}^{n}), \cdots, f((x_m^i)_{i=1}^{n}))
		\end{array}$$
		
		Therefore $\mathfrak{A} = (\mathbb{R}^m, \Phi^{(m)})$ is the product $\mathcal{C}^{\infty}-$ring. In this $\mathcal{C}^{\infty}-$ring we have, in particular, the following binary operation:
		
		$$\begin{array}{cccc}
			\Phi^{(m)}(\cdot ): & \mathbb{R}^m \times \mathbb{R}^m & \rightarrow & \mathbb{R}^m \\
			&((x_j)_{1\leq j \leq m}, (y_j)_{1\leq j \leq m}) & \mapsto & (x_1\cdot y_1, \cdots, x_m \cdot y_m)
		\end{array}$$
		
		\noindent so we write $(x_1, \cdots, x_m) \cdot  (y_1, \cdots, y_m) = (x_1\cdot y_1, \cdots, x_m \cdot y_m)$.
		
		%{\color{red} o que eu quis codificar acima foi:}
		
		%$$\begin{array}{cccc}
			%  \Phi(f): & (\mathbb{R}^m)^n & \rightarrow & \mathbb{R}^m \\
			%    &((x_1^1, \cdots, x_m^1), \cdots, (x^n_1, \cdots, x^n_m)) & \mapsto & (f(x_1^1,\cdots, x^n_1), \cdots, f(x_m^1, \cdots, x_m^n))
			%\end{array}$$
			
			We claim that $\mathfrak{A} = (\mathbb{R}^m, \Phi^{(m)})$ is a von Neumann-regular $\mathcal{C}^{\infty}-$ring. In fact, given any $(a_1, \cdots, a_m)\in\mathbb{R}^m$, then for each $i \in \{1,\cdots, m\}$ such that $a_i \neq 0$, we take $x_i=a_i^{-1}$, and for each $i$ such that $a_i =0$,  we take $x_i = 0$. The element $x= (x_i)_{i=1}^{m} \in \mathbb{R}^m$ is such that:
			
			$$(a_1, \cdots, a_m) = (a_1^2, \cdots, a_m^2)\cdot (x_1, \cdots, x_m).$$
			
			Thus, $\mathfrak{A} = (\mathbb{R}^m, \Phi^{(m)})$ is a von Neumann-regular $\mathcal{C}^{\infty}-$ring.
		\end{example}
		
		\begin{remark}Observe that the construction given in the example above can be replicated by replacing $\mathbb{R}$ by any $\mathcal{C}^{\infty}-$field.    
		\end{remark}

		\begin{remark} 
			Let $\mathbb{T}^{'}$ be the a theory in the language $\mathcal{L}' = \mathcal{L} \cup \{ *\}$, where $*$ is an 1-ary function symbol, which contains:\\
			
			$\bullet$ the (equational) $\mathcal{L}$-axioms for of $\mathcal{C}^{\infty}-$rings; \\
			
			$\bullet$ the (equational) $\mathcal{L}'$-axiom
			$$\sigma:= (\forall x)((x\cdot x^{*}\cdot x = x)\& (x^*\cdot x\cdot x^*=x^*))$$
			% = (\forall x \in A)(\exists ! x^{*} \in A)\varphi(x,x^*),$$
			
			\noindent that is, $\mathbb{T}^{'}:= \mathbb{T}\cup \{ \sigma\}$. By the \textbf{Theorem of Extension by Definition} %(cf. \textbf{Corollary 4.4.7} of \cite{vanD})
			, we know that $\mathbb{T}^{'}$ is a conservative extension of $\mathbb{T}$.
		\end{remark}
		
		%We give a proof of the equivalences above in \textbf{Proposition \ref{proposition1}}.
		
		%From now on we must write, when there is no danger of confusion, $A$ instead of $\mathfrak{A}$.
		
		An homomorphism of von Neumann $\mathcal{C}^{\infty}-$rings, $A$ and $B$ is simply a $\mathcal{C}^{\infty}-$homo\-morphism between these $\mathcal{C}^{\infty}-$rings. We have the following: 
		
		%Now we aim at proving that for any ideal $I \subseteq \mathcal{C}^{\infty}(\mathbb{R}^n)$ its $\mathcal{C}^{\infty}-$radical is also an ideal.\\

		\begin{definition}We denote by $\mathcal{C}^{\infty}{\rm \bf vNRing}$ the category whose objects are von Neumann-regular $\mathcal{C}^{\infty}-$rings and whose morphisms are the $\mathcal{C}^{\infty}-$\-homo\-morphisms between them. Thus, $\mathcal{C}^{\infty}{\rm \bf vNRing}$ is a full subcategory of $\mathcal{C}^{\infty}{\rm \bf Ring}$.
		\end{definition}
		
		The following lemma tells us that, in $\mathcal{C}^{\infty}{\rm \bf vNRing}$, taking localizations and taking the ring of fractions with respect to a special element yields, up to isomorphisms, the same object.
		
		\begin{lemma}\label{zeus}If $A$ is a von Neumann regular $\mathcal{C}^{\infty}-$ring, then given any $a \in A$ there is some idempotent element $e \in A$ such that $A\{ a^{-1}\} \cong A\{ e^{-1}\} \cong A/(1-e)$.
		\end{lemma}
		\begin{proof}Since $a \cdot y = e$ and $e \cdot x = a$, then $e \in \{ a\}^{\infty-{\rm sat}}$ and $a \in \{ e\}^{\infty-{\rm sat}}$. Thus, $\{a\}^{\infty-{\rm sat}} = \{ e\}^{\infty-{\rm sat}}$, so $A\{ a^{-1}\} \cong A\{ e^{-1}\}$. By \textbf{Lemma \ref{destino}}, it follows that $A\{a^{-1}\} \cong A\{ e^{-1}\} \cong A/(1-e)$. 
		\end{proof}

		\begin{lemma}Let $A$ be a von Neumann-regular $\mathcal{C}^{\infty}-$ring, $S \subseteq A$ and let $\widetilde{U}: \mathcal{C}^{\infty}{\rm \bf Ring} \rightarrow {\rm \bf CRing}$ be the forgetful functor. Then:
			$$\widetilde{U}\left( A\{ S^{-1}\}\right) = \widetilde{U}(A)[S^{-1}]$$
		\end{lemma}
		\begin{proof}
			We prove the result first in the case $S=\{ a\}$ for some  $a \in A$. Since $A$ is a von Neumann-regular $\mathcal{C}^{\infty}-$ring, by \textbf{Lemma \ref{zeus}}, given  $a \in A$ there is some idempotent element $e \in A$ such that $(a)=(e)$ and $A\{ a^{-1}\} \cong A\{ e^{-1}\} \cong A/(1-e)$. Now, $A/(1-e) \cong A[e^{-1}]$, and $A[e^{-1}] \cong \widetilde{U}(A)[e^{-1}]$, and since $\widetilde{U}(A)[e^{-1}]$ $\cong \widetilde{U}(A)[a^{-1}]$, as ordinary commutative rings \footnote{ note that $\widetilde{U}(A/(1-e)) = \widetilde{U}(A)/(1-e) = \widetilde{U}(A)[ e^{-1}]$ } the result follows.
			
			Whenever $S$ is finite, we have $A\{ {S}^{-1}\} = A\{ a^{-1}\}$, for $a = \prod S$, and we can use the proof we have just made. For a general $S \subseteq A$, we write $S = \cup_{S' \subseteq_{\rm fin} S} S'$ and use the fact that $\widetilde{U}: \mathcal{C}^{\infty}{\rm \bf Ring} \rightarrow {\rm \bf CRing}$ preserves directed colimits and that $A\{ S^{-1}\} \cong \varinjlim_{S' \subset_{\rm fin} S} A\{ {S'}^{-1}\}$.
		\end{proof}
		
		As a corollary, we have the following:
		
		\begin{prop}\label{vNRingsClosedUnderLocalizations}
			$\mathcal{C}^{\infty}{\rm \bf vNRing} \subseteq \mathcal{C}^{\infty}{\rm \bf Ring}$ is closed under localizations.
		\end{prop}
		
		The following result is an adaptation of \textbf{Proposition 1} of \cite{arndt2016neumann} for the $\mathcal{C}^{\infty}-$case.
		
		\begin{theorem}\label{backy}If $A$ is a von Neumann regular $\mathcal{C}^{\infty}-$ring then $A$ is a reduced $\mathcal{C}^{\infty}-$ring.
		\end{theorem}
		\begin{proof}By the \textbf{Lemma \ref{zeus}},
			$\sqrt[\infty]{(0)} = \{ a \in A \mid A\{ a^{-1}\} \cong \{ 0\}\}$. Now, let $a \in \sqrt[\infty]{(0)}$ and let $e \in A$ be an idempotent element such that $(a)=(e)$. Then $A/(1-e) \cong A\{ e^{-1}\} \cong A\{ a^{-1}\}$. Thus, $A/(1-e) \cong \{ 0\}$ yields $1 \in (1-e)$, so there must exist some $z \in A$ such that $1 = z\cdot (1-e)$, and $(1-e)$ is an invertible idempotent of $A$, so $1-e=1$ and $e=0$. Thus, $a = 0$, so $\sqrt[\infty]{(0)} \subseteq \{ 0\}$. 
		\end{proof}
		
		The following result shows us that whenever $A$ is a von Neumann regular $\mathcal{C}^{\infty}-$ring, the notions of $\mathcal{C}^{\infty}-$spectrum, Zariski spectrum, maximal spectrum and thus, the \index{structure sheaf}structure sheaf of its affine scheme coincide.

		\begin{theorem}\label{bender}Let $A$ be a von Neumann regular $\mathcal{C}^{\infty}-$ring. Then:
			\begin{itemize}
				\item[1)]{$\sqrt[\infty]{(0)} = \sqrt{(0)}=(0)$;}
				\item[2)]{${\rm Spec}^{\infty}\,(A) = {\rm Specm}\,(\widetilde{U}(A)) = {\rm Spec}\, (\widetilde{U}(A))$, as topological spaces;}
				\item[3)]{The structure sheaf of $A$ in the category $\mathcal{C}^{\infty}{\rm \bf Ring}$ is equal to the structure sheaf of $U(A)$ in the category ${\rm \bf CRing}$.}
			\end{itemize}
		\end{theorem}
		\begin{proof}
			Ad 1): By \textbf{Theorem \ref{backy}}, since $A$ is a von Neumann regular $\mathcal{C}^{\infty}-$ring, $\sqrt[\infty]{(0)} = (0)$. Since we always have $(0) \subseteq \sqrt{(0)} \subseteq \sqrt[\infty]{(0)}$, it follows that $\sqrt{(0)}=(0)$.

			%Now, given $a \in \sqrt{(0)}$, there is some $n \in \mathbb{N}$ such that $a^n = 0$, so $0 = a^n = (ae)^n = a^n e^n = a^ne$. Since $e^n = a^n y^n = 0$, then $e=0$, so $a = a \cdot 0 = 0$. Thus, $\sqrt{(0)} = (0) = \sqrt[\infty]{(0)}$.

			Ad 2): Note that in a von Neumann regular $\mathcal{C}^{\infty}-$ring every prime ideal is a maximal ideal. In fact, let $\mathfrak{p}$ be a prime ideal in $A$. Given $a + \mathfrak{p} \neq \mathfrak{p}$ in $A/\mathfrak{p} $, then $a + \mathfrak{p} \in \left( A/\mathfrak{p} \right)^{\times}$. Since $A$ is a von Neumann regular ring, there exists some $b \in A$ such that $aba = a$, so $a + \mathfrak{p} = aba + \mathfrak{p}$, $a + \mathfrak{p} = (ab + \mathfrak{p})\cdot (a + \mathfrak{p})$, $ab + \mathfrak{p} = 1 + \mathfrak{p}$ and, thus, $ab=1$.
			
			Hence, every non-zero element of $A/\mathfrak{p}$ is invertible, so $A/\mathfrak{p}$ is a field. Under those circumstances, it follows that $\mathfrak{p}$ is a maximal ideal, so ${\rm Spec}\, (A) = {\rm Specm}\,(A)$.
			
			We always have ${\rm Specm}\, (A) \subseteq {\rm Spec}^{\infty}\, (A)$ and ${\rm Spec}^{\infty}\, (A) \subseteq {\rm Spec}\,(A)$, so:
			$${\rm Spec}\,(A) \subseteq {\rm Specm}\, (A) \subseteq {\rm Spec}^{\infty}\,(A) \subseteq {\rm Spec}\,(A)$$
			and ${\rm Spec}^{\infty}\,(A) = {\rm Spec}\, (A)$.
			
			Note, also, that both the topological spaces ${\rm Spec}\,(A)$ and ${\rm Spec}^{\infty}\,(A)$ have the same basic open sets, $D^{\infty}(a) = \{ \mathfrak{p} \in {\rm Spec}^{\infty}\,(A) \mid a \notin \mathfrak{p} \} = \{ \mathfrak{p} \in {\rm Spec}\,(A) \mid a \notin \mathfrak{p}\} = D(a)$, hence ${\rm Spec}^{\infty}\,(A) = {\rm Spec}\, (A)$ also as topological spaces.
			
			Ad 3). It is an immediate consequence of $2$.
		\end{proof}
		
		\begin{prop}\label{proposition1}Let $A$ be a $\mathcal{C}^{\infty}-$ring. Then the following are equivalent:
			\begin{itemize}
				\item[(i)]{$A$ is von Neumann-regular, \textit{i.e.}, $(\forall a \in A)(\exists x \in A)(a = a^2x)$.}
				\item[(ii)]{Every principal ideal of $A$ is generated by an idempotent element, \textit{i.e.},
					$$(\forall a \in A)(\exists e \in A)(\exists y \in A)(\exists z \in A)((e^2=e)\&(ey=a)\wedge(az=e))$$}
				\item[(iii)]{$(\forall a \in A)(\exists ! b \in A )((a=a^2b)\&(b=b^2a))$}
			\end{itemize}
			Moreover, when $A$ is von Neumann-regular, then $A$ is ($\mathcal{C}^{\infty}-$)reduced (\textit{i.e.}, $\sqrt[\infty]{(0)}= \sqrt{(0)} =(0)$) and for each $a \in A$ the idempotent element $e$ satisfying (ii) and the element $b$ satisfying (iii) are uniquely determined.
		\end{prop}
		%\begin{proof}
		%See \textbf{Proposition 2} of \cite{berni2019neumann}.
		%\end{proof}

		\begin{proof}
			The implication (iii) $\to$ (i) is obvious, so we omit the proof.
			
			Ad (i) $\to$ (ii): Let $I = (a)$ be a principal ideal of $A$. By (i), there is $x \in A$ such that $a = a^2x$, so we define $e:=ax$, which is idempotent since $e^2 = (ax)^2 = a^2x^2 = (a^2x)x = ax = e$. By definition, $e = ax \in (a) = I$, so $(e) \subseteq I$, and since  $a = a^2x = (ax)a = ea$ we also have $a \in (e)$, so $I = (a) \subseteq (e)$. Hence, $I = (e)$.
			
			Ad (ii) $\to$ (i): Let $a \in A$ be any element. By (ii) there are $e \in A$, $y \in A$ and $z \in A$ such that $e^2 =e$, $a = ey$ and $e = az$. Define $x:= z^2y$, and we have $a^2x = a^2z^2y = e^2y = ey = a$.
			
			Ad (i) $\to$ (iii): Let $a \in A$ be any element. By (i), there is some $x \in A$ such that $a= a^2x$. There can be many $x \in A$ satisfying this role, but there is a ``minimal'' one: the element $ax$ is idempotent and we can project any chosen $x$ down with this idempotent, obtaining $b:=ax^2$. Then $aba = aab^2a = (ax)(ax)a = axa = a$ and $bab = (ax^2)a(ax^2) = (ax)^3x = (ax)x = b$.

			Now suppose that $A$ is a von Neumann-regular $\mathcal{C}^{\infty}-$ring, and let $a \in A$ be such that $a \in \sqrt[\infty]{(0)}$. Then let $e$ be an idempotent such that $ey=a$, $az=e$, for some $y,z \in A$. Then $a$ is such that $A\{ a^{-1}\} \cong \{0\}$, and by \textbf{Lemma \ref{zeus}} there is some idempotent $e \in A$ such that $A\{ a^{-1}\} \cong A/(1-e)$. Now, $A\{ a^{-1}\} \cong \{ 0\}$ occurs if and only if, $A/(1-e) \cong \{0\}$, \textit{i.e.}, if and only if, $(1-e) = A$. Since $(1-e) = A$, it follows that $1-e \in A^{\times}$, and since $e \cdot (1-e) = 0$, it follows by cancellation that $e = 0$, hence $a = ey = 0y=0$.
			
			Let $e,e' \in A$ be idempotents of an arbitrary ring satisfying $(e) = (e')$. Select $r,r'\in A$ such that $er'=e'$ and $e'r = e$. Then $e' = er' = er'e = e'e = e're'=e'r=e$. Thus, if an ideal is generated by an idempotent element, this element is uniquely determined.
			
			Finally, let $A$ be a von Neumann-regular $\mathcal{C}^{\infty}-$ring. Select a member $a \in A$ and consider $b,b'\in A$ such that $a^2b'=a=a^2b$, $b=b^2a$, $b'={b'}^2a$. Then $(b-b')a^2=(b-b')a^2=(b-b')(ba^2-b'a^2)=(b-b')(a-a)=(b-b')\cdot 0 = 0$ and $[(b-b')\cdot a]^2 \in (0)$. Since $A$ is $\mathcal{C}^{\infty}-$reduced, $[(b-b')\cdot a]^2 \in (0) = \sqrt[\infty]{(0)}$. By item (1) of \textbf{Theorem \ref{bender}}, $\sqrt[\infty]{(0)} = \sqrt{(0)}$, so $[(b-b')\cdot a]^2 \in \sqrt[\infty]{(0)} = \sqrt{(0)}$ and $(b-b')\cdot a = 0$. Therefore $b - b' = b^2a - {b'}^2a= (b^2-{b'}^2)a = (b+b')(b-b')a = (b+b') \cdot 0 = 0$.
		\end{proof}

		\begin{remark}\label{agp}Let $A$ be a von Neumann-regular $\mathcal{C}^{\infty}-$ring and $e \in A$ be any idempotent element. Then $A \cdot e$ is a von Neumann-regular $\mathcal{C}^{\infty}-$ring. Indeed, we have $A \cdot e \cong A/(1-e)$ and the latter is an homomorphic image of a von Neumann-regular $\mathcal{C}^{\infty}-$ring, namely $A/(1-e) = q[A]$. %Since $A/(1-e)$ is the homomorphic image of a von Neumann-regular $\mathcal{C}^{\infty}-$ring, $A/(1-e)$ is a von Neumann-regular $\mathcal{C}^{\infty}-$ring. Since $A \cdot e \cong A/(1-e)$, it follows that $A \cdot e$ is a von Neumann-regular $\mathcal{C}^{\infty}-$ring.
		\end{remark}

		\begin{lemma}\label{moura}Let $A$ be a local $\mathcal{C}^{\infty}-$ring. The only idempotent elements of $A$ are $0$ and $1$.
		\end{lemma}
		\begin{proof}
			See \textbf{Lemma 4} of \cite{berni2019neumann}.
		\end{proof}
		%\begin{proof}(Sketch)
		%Since $A$ is a local $\mathcal{C}^{\infty}-$ring, we have $(\forall x \in A)(\forall y \in A)((x+y \in A^{\times}) \to (x \in A^{\times})\vee(y \in A^{\times})).$ Let $e \in A$ be an idempotent. We have $(1 = e + (1-e) \in A^{\times})\to((e \in A^{\times})\vee(1-e \in A^{\times}))$. If $e \in A^{\times}$ then $e=1$ and if $1-e \in A^{\times}$ then $1-e = 1$, so $e=0$.
		%\end{proof}

		\begin{prop}\label{dan}Let $A$ be a von Neumann-regular $\mathcal{C}^{\infty}-$ring whose only idempotent elements are $0$ and $1$. Then the following assertions are equivalent:
			\begin{itemize}
				\item[(i)]{$A$ is a $\mathcal{C}^{\infty}-$field;}
				\item[(ii)]{$A$ is a $\mathcal{C}^{\infty}-$domain;}
				\item[(iii)]{$A$ is a local $\mathcal{C}^{\infty}-$ring.}
			\end{itemize}
		\end{prop}
		%\begin{proof}
		%See \textbf{Proposition 3} of \cite{berni2019neumann}.
		%\end{proof}

		\begin{proof}
			The implications (i) $\to$ (ii), (i) $\to$ (iii) are immediate, so we omit their proofs. 
			
			Ad (iii) $\to$ (i): Suppose $A$ is a local $\mathcal{C}^{\infty}-$ring. Since $A$ is a von Neumann-regular $\mathcal{C}^{\infty}-$ring, given any $x \in A\setminus \{ 0\}$ there exists some idempotent element $e \in A$ such that $(x)=(e)$. However, the only idempotent elements of $A$ are, by \textbf{Lemma \ref{moura}}, $0$ and $1$. We claim that $(x)=(1)$, otherwise we would have $(x)=(0)$, so $x=0$.
			
			Now, $(x) = (1)$ implies $1 \in (x)$, so there is some $y \in A$ such that $1 = x \cdot y = y \cdot x$, and $x$ is invertible. Thus $A$ is a $\mathcal{C}^{\infty}-$field.
			
			Ad (ii) $\to$ (i): Suppose $A$ is a $\mathcal{C}^{\infty}-$domain. Given any $x \in A \setminus \{ 0\}$, we have $(\forall y \in A \setminus \{ 0\})(x\cdot y \neq 0)$, so $(x)\neq (0)$. Since $A$ is a von Neumann-regular $\mathcal{C}^{\infty}-$ring, $(x)$ is generated by some non-zero idempotent element, namely, $1$. Hence $(x) = (1)$ and $x \in A^{\times}$.
		\end{proof}

		\begin{prop}\label{proposition4} The inclusion functor $\imath : \mathcal{C}^{\infty}{\rm \bf vNRing} \hookrightarrow \mathcal{C}^{\infty}{\rm \bf Ring}$ creates filtered colimits, \textit{i.e.}, $\mathcal{C}^{\infty}{\rm \bf vNRing}$ is closed  in $\mathcal{C}^{\infty}{\rm \bf Ring}$ under filtered colimits.
		\end{prop}
		%\begin{proof}
		%See \textbf{Proposition 4} of \cite{berni2019neumann}.
		%\end{proof}

		\begin{proof}(Sketch) Filtered colimits in $\mathcal{C}^{\infty}{\rm \bf Ring}$ are formed by taking the colimit of the underlying sets and defining the $n-$ary functional symbol $f^{(n)}$ of an $n-$tuple $(a_1, \cdots, a_n)$ into a common $\mathcal{C}^{\infty}-$ring occurring in the diagram and taking the element $f^{(n)}(a_1, \cdots, a_n)$ there. Given a filtered poset $(I, \leq)$ and a a filtered family of $\mathcal{C}^{\infty}-$rings, for every element $\alpha \in \varinjlim A_i$, there is some $i \in I$ and $a_i \in A_i$ such that $\alpha = [(a_i,i)]$. Since $A_i$ is a von Neumann-regular $\mathcal{C}^{\infty}-$ring, there must exist some idempotent $e_i \in A_i$ such that $(a_i)=(e_i)$. It suffices to take $\eta = [(e_i,i)] \in \varinjlim A_i$, which is an idempotent element of $\varinjlim A_i$ such that $(\alpha) = ([(a_i, i)]) = ([(e_i,i)]) = (\eta)$.
		\end{proof}

		We have the following important result, which relates von Neumann-regular $\mathcal{C}^{\infty}-$rings to the topology of its smooth Zariski spectrum:
		
		\begin{theorem}\label{ota} Let $A$ be a $\mathcal{C}^{\infty}-$ring. The following assertions are equivalent:
			\begin{itemize}
				\item[(i)]{$A$ is a von Neumann-regular $\mathcal{C}^{\infty}-$ring;}
				\item[(ii)]{$A$ is a $\mathcal{C}^{\infty}-$reduced $\mathcal{C}^{\infty}-$ring (i.e., $\sqrt[\infty]{(0)}=(0)$) and ${\rm Spec}^{\infty}\,(A)$ is a Boolean space, \textit{i.e.},a compact, Hausdorff and totally disconnected space.}
			\end{itemize}
		\end{theorem}
		\begin{proof}
			Ad $(i) \to (ii)$: it follows from item (1) of \textbf{Theorem \ref{bender}}.

			Since ${\rm Spec}^{\infty}\,(A)$ is a spectral space, we only need to show that $\mathcal{B} = \{ D^{\infty}(a) \mid a \in A \}$ is a clopen basis for its topology.
			
			Given any $a \in A$, since $A$ is a von Neumann regular $\mathcal{C}^{\infty}-$ring, there is some idempotent element $e \in A$ such that $(a)=(e)$, so $D^{\infty}(a) = D^{\infty}(e)$. We claim that ${\rm Spec}^{\infty}\,(A) \setminus D^{\infty}(e) = D^{\infty}(1-e)$, hence $D^{\infty}(a) = D^{\infty}(e)$ is a clopen set.
			
			In fact, from item (iii) of \textbf{Lemma 1.2} of \cite{rings2}, $D^{\infty}(e) \cap D^{\infty}(1-e) = D^{\infty}(e\cdot(1-e)) = D^{\infty}(0) = \{ \mathfrak{p} \in {\rm Spec}^{\infty}\,(A) \mid 0 \notin \mathfrak{p} \} = \varnothing$.  Moreover, for every prime ideal $\mathfrak{p}$ we have $e \notin \mathfrak{p}$ or $(1-e) \notin \mathfrak{p}$ (for if this was not the case, we would have a prime ideal $\mathfrak{p}_0$ such that $e \in \mathfrak{p}_0$ and $(1-e) \in \mathfrak{p}_0$, so $1 = (1-e)+e \in \mathfrak{p}_0$, which would not be a proper ideal). Thus ${\rm Spec}^{\infty}(A) = D^{\infty}(e) \cup D^{\infty}(1-e)$.
			
			%Also by item (iii) of \textbf{Lemma 1.2} of \cite{rings2}, we have $D^{\infty}(e) \cup D^{\infty}(1-e) = D^{\infty}(e + (1-e)) = D^{\infty}(1) = {\rm Spec}^{\infty}\,(A)$, so it follows that ${\rm Spec}^{\infty}\,(A)$ is a Boolean space and $A$ is a $\mathcal{C}^{\infty}-$reduced $\mathcal{C}^{\infty}-$ring.
			
			Ad $(ii) \to (i)$. Since ${\rm Spec}^{\infty}\,(A)$ is a Boolean space, it is a Hausdorff space and for every $a \in A$, $D^{\infty}(a)$ is compact, hence it is closed. Thus, we conclude that for every $a \in A$, $D^{\infty}(a)$ is a clopen set, and so  ${\rm Spec}^{\infty}\,(A) \setminus D^{\infty}(a)$ is a clopen subset of ${\rm Spec}^{\infty}\,(A)$.
			
			Now, for every clopen $C$ in ${\rm Spec}^{\infty}\,(A)$ there is some $b \in A$ such that $C = D^{\infty}(b)$. Since $C$ is clopen in ${\rm Spec}^{\infty}\,(A)$, it is in particular an open set, and since $\{ D^{\infty}(a) \mid a \in A\}$ is a basis for the topology of ${\rm Spec}^{\infty}\,(A)$,  there is a family, $\{ b_i\in A \mid i \in I\}$, of elements of $A$ such that $C = \cup_{i \in I} D^{\infty}(b_i)$. Since $C$ is compact, there is a finite subset $I' \subseteq I$ such that $C = \cup_{i \in I'} D^{\infty}(b_i)$. Applying the item (iii) of \textbf{Lemma 1.4} of \cite{rings2}, we conclude that there is some element $b \in A$ such that $\bigcup_{i \in I'} D^{\infty}(b_i) = D^{\infty}(b).$
			
			Since ${\rm Spec}^{\infty}\,(A) \setminus D^{\infty}(a)$ is clopen, there is some $d \in A$ such that ${\rm Spec}^{\infty}\,(A) \setminus D^{\infty}(a) = D^{\infty}(d)$. We have $\varnothing = D^{\infty}(a) \cap D^{\infty}(d) = D^{\infty}(a\cdot d) = \{ \mathfrak{p} \in {\rm Spec}^{\infty}\,(A) \mid a \cdot d \notin \mathfrak{p} \}$, so $(\forall \mathfrak{p} \in {\rm Spec}^{\infty}\,(A))(a \cdot d \in \mathfrak{p})$,
			hence $a \cdot d \in \bigcap {\rm Spec}^{\infty}\,(A) = \sqrt[\infty]{(0)} = (0)$, where the last equality is due to the fact that $A$ is a $\mathcal{C}^{\infty}-$reduced $\mathcal{C}^{\infty}-$ring.
			
			We have, then, $a \cdot d = 0.$ Also, we have $D^{\infty}(a^2 + d^2) = D^{\infty}(a) \cup D^{\infty}(d) = {\rm Spec}^{\infty}\,(A) = D^{\infty}(1)$. By item (i) of \textbf{Lemma 1.4} of \cite{rings2}, $D^{\infty}(a^2+d^2) \subseteq D^{\infty}(1)$ implies $a^2+d^2 \in \{ 1\}^{\infty-{\rm sat}}$. Since $\{ 1\}^{\infty-{\rm sat}} = A^{\times}$, it follows that $a^2+d^2 \in A^{\times}$, so there is some $y \in A$ such that $y \cdot (a^2+d^2) = 1, ya^2 + yd^2 = 1$. Since $a \cdot d = 0$, we get $a(a^2y)+a(b^2y)= a \cdot 1 = a$$
			$$a^2 (a \cdot y) = a$. %, and considering the forgetful functor $\widetilde{U}: \mathcal{C}^{\infty}{\rm \bf Ring} \to \, {\rm \bf CRing}$.
		\end{proof}

		%%%%%%%%%%%%%%%%%%%%%%%%%%%%%%%%%%%%%%%%%%%%%%%%%%%%%%%%%%%%%
		The following proposition will be useful to characterize the von Neumann-regular $\mathcal{C}^{\infty}-$rings by means of the ring of global sections of the structure sheaf of its affine scheme.
		
		\begin{prop}\label{Phil} If a $\mathcal{C}^{\infty}-$ring $A$ is a von-Neumann-regular $\mathcal{C}^{\infty}-$ring and $\mathfrak{p} \in {\rm Spec}^{\infty}\,(A)$, then ${A}/{\mathfrak{p}} \cong A\{ {A \setminus \mathfrak{p}}^{-1}\}$ and both are $\mathcal{C}^{\infty}-$fields.
		\end{prop}
		%\begin{proof}See \textbf{Proposition 5} of \cite{berni2019neumann}.
		%\end{proof}
		\begin{proof}
			We are going to show that the only maximal ideal of $A\{ {A \setminus \mathfrak{p}}^{-1}\}$, $\mathfrak{m}_{\mathfrak{p}}$ is such that $\mathfrak{m}_{\mathfrak{p}} \cong \{ 0\}$.
			
			Let $\eta_{\mathfrak{p}}: A \to A\{ {A \setminus \mathfrak{p}}^{-1}\}$ be the localization morphism of $A$ with respect to $A \setminus \mathfrak{p}$. We have $\mathfrak{m}_{\mathfrak{p}} = \langle \eta_{\mathfrak{p}}[A \setminus \mathfrak{p}]\rangle = \left\{ {\eta_{\mathfrak{p}}(a)}/{\eta_{\mathfrak{p}}(b)} \mid (a \in \mathfrak{p})\&(b \in A \setminus \mathfrak{p})  \right\}$. We must show that for every $a \in \mathfrak{p}$, $\eta_{\mathfrak{p}}(a) = 0$, which is equivalent, by \textbf{Theorem 1.4} of \cite{rings2}, to assert that for every $a \in \mathfrak{p}$ there is some $c \in (A \setminus \mathfrak{p})^{\infty-{\rm sat}} = A\setminus \mathfrak{p}$  such that $c \cdot a = 0$ in $A$.
			
			\textit{Ab absurdo}, suppose $\mathfrak{m}_{\mathfrak{p}} \neq \{ 0\}$, so there is $a \in \mathfrak{p}$ such that $\eta_{\mathfrak{p}}(a)\neq 0$, i.e., such that for every $c \in A \setminus \mathfrak{p}$, $c \cdot a \neq 0$. Since $A$ is a von Neumann-regular $\mathcal{C}^{\infty}-$ring, for this $a$ there is some idempotent $e \in \mathfrak{p}$ such that $(a) = (e)$.
			
			Since $a \in (e)$, there is some $\lambda \in A$ such that $a = \lambda \cdot a$, hence:
			$$0 \neq \eta_{\mathfrak{p}}(a) = \eta_{\mathfrak{p}}(\lambda \cdot e) = \eta_{\mathfrak{p}}(\lambda) \cdot \eta_{\mathfrak{p}}(e)$$
			and $\eta_{\mathfrak{p}}(e) \neq 0$.
			
			Since $\eta_{\mathfrak{p}}(e) \neq 0$,
			\begin{equation}\label{eq1}
				(\forall d \in A\setminus \mathfrak{p})(d \cdot e \neq 0).
			\end{equation}
			
			Since $e$ is an idempotent element, $1-e \notin \mathfrak{p}$, for if $1-e \in \mathfrak{p}$ then $e + (1 - e) = 1 \in \mathfrak{p}$ and $\mathfrak{p}$ would not be a proper prime ideal.
			
			We have also:
			\begin{equation}\label{eq2}
				e \cdot (1 - e) = 0,
			\end{equation}
			
			The equation \eqref{eq2} contradicts \eqref{eq1}, so $\mathfrak{m}_{\mathfrak{p}} \cong \{ 0\}$ and $A\{{A \setminus \mathfrak{p}}^{-1}\}$ is a $\mathcal{C}^{\infty}-$field.
		\end{proof}

		As a consequence, we register another proof of $(iii)  \rightarrow (i)$ of {\bf Proposition \ref{dan}}.
		
		\begin{cor}\label{fischer}Let $\mathfrak{A}=(A,\Phi)$ be a local von Neumann-regular $\mathcal{C}^{\infty}-$\-ring. Then $\mathfrak{A}$ is a $\mathcal{C}^{\infty}-$field.
		\end{cor}
		\begin{proof}(Sketch)Let  $\mathfrak{m} \subseteq A$ be the unique maximal ideal of $A$. By \textbf{Proposition \ref{Phil}}, since $A$ is  von Neumann-regular, $A_{\mathfrak{m}} \cong {A}/{\mathfrak{m}}$, which is a $\mathcal{C}^{\infty}-$field.
			Also, $A_{\mathfrak{m}} = A\{ A \setminus \mathfrak{m}^{-1}\} = A\{ {A^{\times}}^{-1}\} \cong A$, and since $A_{\mathfrak{m}}$ is isomorphic to a $\mathcal{C}^{\infty}-$field, it follows that $A$ is a $\mathcal{C}^{\infty}-$field.
		\end{proof}

		Summarizing, we have the following result:
		
		\begin{theorem}\label{lab}If $A$ is a von Neumann-regular $\mathcal{C}^{\infty}-$ring, then the set ${\rm Spec}^{\infty}\,(A)$ with the smooth Zariski topology, ${\rm Zar}^{\infty}$, is a Boolean topological space, by  \textbf{Theorem \ref{ota}}. Moreover, by \textbf{Proposition \ref{Phil}}, for every $\mathfrak{p} \in {\rm Spec}^{\infty}\,(A)$,
			$$A_{\mathfrak{p}} = \varinjlim_{a \notin \mathfrak{p}} A\{ a^{-1}\} \cong A\{ A\setminus \mathfrak{p}^{-1}\}$$
			is a $\mathcal{C}^{\infty}-$field.
		\end{theorem}

		The above theorem suggests us that von Neumann-regular $\mathcal{C}^{\infty}-$rings behave much like ordinary von Neumann-regular commutative unital rings. In the next sections we are going to explore this result using sheaf theoretic machinery.

		\begin{prop}\label{proposition6}The limit in $\mathcal{C}^{\infty}{\rm \bf Ring}$ of a diagram of von Neumann-regular $\mathcal{C}^{\infty}$\-rings is a von Neumann-regular $\mathcal{C}^{\infty}-$ring. In particular, $\mathcal{C}^{\infty}{\rm \bf vNRing}$ is a complete category and the inclusion functor from the category of all von Neumann regular $\mathcal{C}^{\infty}-$rings, $\mathcal{C}^{\infty}{\rm \bf vNRing}$, to $\mathcal{C}^{\infty}{\rm \bf Ring}$  preserves all limits.
		\end{prop}
		\begin{proof}(Sketch) It is clear from the definition that the class $\mathcal{C}^{\infty}{\rm \bf vNRing}$ of von Neumann-regular $\mathcal{C}^{\infty}-$rings is closed under arbitrary products in the class $\mathcal{C}^{\infty}{\rm \bf Ring}$, of all $\mathcal{C}^{\infty}-$rings. Thus it suffices to show that it is closed under equalizers. 
			
			So let $A, B$ be von Neumann-regular rings and $ f,g: A \to B$ be
			$\mathcal{C}^{\infty}-$homomorphisms. Their equalizer in $\mathcal{C}^{\infty}{\rm \bf Ring}$ is given by the set $E = \{ a \in A \mid f(a)=g(a)\}$, endowed with the restricted ring operations from $A$.
			
			To see that E is von Neumann-regular, we need to show that for $a \in E$, the (unique) element $b$ satisfying  $ab^2 = b$ and $a^2b = a$ also belongs to $E$. But this is true since the quasi-inverse element is unique and is preserved under $\mathcal{C}^{\infty}-$homomorphisms.
		\end{proof}
		
		\begin{prop}\label{proposition8} The category $\mathcal{C}^{\infty}{\rm \bf vNRing}$ is the smallest subcategory of $\mathcal{C}^{\infty}{\rm \bf Ring}$ closed under limits containing all $\mathcal{C}^{\infty}-$fields.
		\end{prop}
		%\begin{proof}
		%See \textbf{Proposition 7} of \cite{berni2019neumann}.
		%\end{proof}
		
		\begin{proof}
			Clearly all $\mathcal{C}^{\infty}-$fields are von Neumann-regular $\mathcal{C}^{\infty}-$rings, and by \textbf{Proposition \ref{proposition6}} so are limits of $\mathcal{C}^{\infty}-$fields. Thus $\mathcal{C}^{\infty}{\rm \bf vNRing}$ contains all limits of $\mathcal{C}^{\infty}-$fields. On the other hand the ring of global sections of a sheaf can be expressed as a limit of a diagram of products and ultraproducts of the stalks (by \textbf{Lemma 2.5} of \cite{Ken}). All these occurring (ultra)products are von Neumann-regular $\mathcal{C}^{\infty}-$rings as well and hence so is their limit, by \textbf{Proposition \ref{proposition6}}.
		\end{proof}

		\section{Von Neumann-regular $\mathcal{C}^{\infty}-$Rings and Boolean Algebras}
		
		%\textcolor{red}{REVISAR}

		In this section we also apply von Neumann regular $\mathcal{C}^{\infty}$-ring to naturally represent Boolean Algebras in a strong sense: i.e., not only all Boolean algebras are isomorphic to the Boolean algebra of idempotents of a von Neumann regular $\mathcal{C}^{\infty}$-ring, as every homomorphism between such Boolean algebras of idempotents is (essentially) induced by a $\mathcal{C}^{\infty}$-homomorphism.

		%We start with the following general results:
		
		%\begin{prop}\label{tio}Let $(L, \wedge, \vee, \leq)$ be any lattice. Then:
		%$$(\forall a \in L)(\forall b \in L)(\forall c \in L)(a \wedge (b \vee c) = (a \wedge b) \vee (a \wedge c))$$
		%if, and only if,
		%$$(\forall a \in L)(\forall b \in L)(\forall c \in L)(a \vee (b \wedge c) = (a \vee b) \wedge (a \vee c))$$
		%\end{prop}
		
		\begin{remark}\label{stone-dual}By Stone Duality, there is an anti-equivalence of categories between the category of Boolean algebras, ${\rm \bf BA}$, and the category of Boolean spaces, ${\rm \bf BoolSp}$.
			
			Under this anti-equivalence, a Boolean space $(X,\tau)$ is mapped to the Boolean algebra of clopen subsets of $(X,\tau)$, ${\rm Clopen}\,(X)$:
			
			$$\begin{array}{cccc}
				{\rm Clopen}: & {\rm \bf BoolSp} & \rightarrow & {\rm \bf BA} \\
				& \xymatrix{(X,\tau) \ar[r]^{f}& (Y,\sigma)} & \mapsto & \xymatrix{{\rm Clopen}\,(Y) \ar[r]^{f^{-1}\upharpoonright}& {\rm Clopen}\,(X)}
			\end{array}$$
			
			The quasi-inverse functor is given by the Stone space functor: a Boolean algebra $B$ is mapped to the Stone space of $B$, ${\rm Stone}(B) = (\{ U \subseteq B: U \,\, \text{is an ultrafilter in} \,\, B\}, \tau_B)$ , where $\tau_B$ is the topology whose basis is given by the image of the map $t_B : B \to \mathscr{P}\,({\rm Stone}(B))$ (the set of all subsets of ${\rm Stone}(B)$), $b \mapsto t_B(b) = S_B(b) = \{ U \in {\rm Stone}(B) : b \in U\}$.
			
			$$\begin{array}{cccc}
				{\rm Stone}: & {\rm \bf BA} & \rightarrow & {\rm \bf BoolSp} \\
				& \xymatrix{B \ar[r]^{h}& B'} & \mapsto & \xymatrix{{{\rm Stone}(B')} \ar[r]^{h^{-1}\upharpoonright}& {{\rm Stone}(B)}}
			\end{array}$$

		\end{remark}
		
		\begin{remark}
			
			Let $(A', +', \cdot', 0', 1')$ be any commutative unital ring, $B(A')= \{ e \in A' \mid e^2 = e \}$ and denote by $\wedge', \vee', *', \leq', 0'$ and $1'$ its respective associated Boolean algebra operations, relations and constant symbols as constructed above. Note that for any commutative unital ring homomorphism $f: A \to A'$, the map $B(f):= f\upharpoonright_{B(A)}: B(A) \to B(A')$ is such that:
			
			\begin{itemize}
				\item[(i)]{$B(f)[B(A)] \subseteq B(A')$;}
				\item[(ii)]{$(\forall e_1 \in A)(\forall e_2 \in A)(B(f)(e_1 \wedge e_2) = f\upharpoonright_{B(A)}(e_1 \cdot e_2) = (f\upharpoonright_{B(A)}(e_1))\cdot' (f\upharpoonright_{B(A)}(e_2))= B(f)(e_1)\wedge'B(f)(e_2))$}
				\item[(iii)]{$(\forall e_1 \in A)(\forall e_2 \in A)(B(f)(e_1 \vee e_2) = f\upharpoonright_{B(A)}(e_1 + e_2 - e_1 \cdot e_2) = (f\upharpoonright_{B(A)}(e_1))+' (f\upharpoonright_{B(A)}(e_2)) - f\upharpoonright_{B(A)}(e_1)\cdot f\upharpoonright_{B(A)}(e_2)= B(f)(e_1)\vee'B(f)(e_2))$}
				\item[(iv)]{$(\forall e \in B(A))(B(f)(e^{*}) = f\upharpoonright_{B(A)}(1-e)=f\upharpoonright_{B(A)}(1)-f\upharpoonright_{B(A)}(e) = 1' - f\upharpoonright_{B(A)}(e) = B(f)(e)^{*})$}
			\end{itemize}
			
			\noindent hence a morphism of Boolean algebras.
			
			We also have, for every ring $A$, $B({\rm id}_A) = {\rm id}_{B(A)}$ and given any $f: A \to A'$ and $f': A' \to A''$, $B(f' \circ f) = B(f')\circ B(f)$, since $B(f) = f\upharpoonright_{B(A)}$, so:
			
			$$\begin{array}{cccc}
				B: & {\rm \bf CRing} & \rightarrow & {\rm \bf BA} \\
				& A & \mapsto & B(A)\\
				& {\xymatrix{A \ar[r]^{f}& A'}} & \mapsto & \xymatrix{B(A) \ar[r]^{B(f)} & B(A')}
			\end{array}$$
			
			is a (covariant) functor.
			
		\end{remark}

		Since we can regard any $\mathcal{C}^{\infty}-$ring $A$ as a commutative unital ring via the forgetful functor $\widetilde{U}: \mathcal{C}^{\infty}{\rm \bf Ring} \to {\rm \bf CRing}$, we have a (covariant) functor:
		
		$$\begin{array}{cccc}
			\widetilde{B}: & \mathcal{C}^{\infty}{\rm \bf Ring} & \rightarrow & {\rm \bf BA} \\
			& A & \mapsto & \widetilde{B}(A):= (B \circ U)(A)\\
			& {\xymatrix{A \ar[r]^{f}& A'}} & \mapsto & \xymatrix{\widetilde{B}(A) \ar[r]^{\widetilde{B}(f)} & \widetilde{B}(A')}
		\end{array}$$

		Now, if $A$ is any $\mathcal{C}^{\infty}-$ring, we can define the following map:
		
		$$\begin{array}{cccc}
			\jmath_A: & \widetilde{B}(A) & \rightarrow & {\rm Clopen}\,({\rm Spec}^{\infty}(A)) \\
			& e & \mapsto & D^{\infty}(e) = \{ \mathfrak{p} \in {\rm Spec}^{\infty}\,(A)\mid e \notin \mathfrak{p}\}
		\end{array}$$
		
		\textbf{Claim 1:} The map defined above is a  Boolean algebra homomorphism.\\
		
		Note that for any $e \in B(A)$, $D^{\infty}(e^*) = D^{\infty}(1-e) = {\rm Spec}^{\infty}\,(A)\setminus D^{\infty}(e) = D^{\infty}(e)^{*}$, since:
		$$D^{\infty}(e) \cap D^{\infty}(1-e) = D^{\infty}\,(e \cdot (1-e)) = D^{\infty}(0) = \varnothing$$
		and
		\begin{multline*}D^{\infty}(e)\cup D^{\infty}(e*) = D^{\infty}(e)\cup D^{\infty}(1-e) = D^{\infty}(e^2 + (1-e)^2) = D^{\infty}(e + (1-e))=\\
			= D^{\infty}(1) = {\rm Spec}^{\infty}\,(A).\end{multline*}
		
		Hence $\jmath_A(e^{*}) = \jmath_A(1-e) = D^{\infty}(1-e) = {\rm Spec}^{\infty}\,(A) \setminus D^{\infty}\,(e) = {D^{\infty}(e)}^{*} = \jmath_A(e)^{*}$.

		By the item (iii) of \textbf{Lemma 1.4} of \cite{rings2}, $D^{\infty}(e \cdot e') = D^{\infty}(e) \cap D^{\infty}(e')$, so $\jmath_A(e \wedge e') = D^{\infty}(e\cdot e') = D^{\infty}(e) \cap D^{\infty}(e') = \jmath_A(e)\cap \jmath_A(e').$
		
		Last,
		\begin{multline*}
			\jmath_A(e \vee e') = \jmath_A(e + e' - e \cdot e') = D^{\infty}(e + e' - e \cdot e') = D^{\infty}(e^2)\cup D^{\infty}(e' - e \cdot e')= \\
			D^{\infty}(e)\cup D^{\infty}(e'\cdot(1 - e)) = D^{\infty}(e)\cup [D^{\infty}(e') \cap D^{\infty}(1-e)] = \\
			= [D^{\infty}(e) \cup D^{\infty}(e')] \cap [D^{\infty}(e) \cup D^{\infty}(1-e)] = [D^{\infty}(e) \cup D^{\infty}(e')]\cap {\rm Spec}^{\infty}\,(A) =\\
			= D^{\infty}(e) \cup D^{\infty}(e') = \jmath_A(e) \cup \jmath_A(e'),
		\end{multline*}

		\noindent and the claim is proved.\\

		\textbf{Claim 2:} $\jmath_A: \widetilde{B}(A) \to \operatorname{Clopen}\,(\operatorname{Spec}^{\infty}(A))$ is an injective map.\\

		In order to prove that $\jmath_A : \widetilde{B}(A) \to {\rm Clopen}\,({\rm Spec}^{\infty}\,(A))$ is an injective map, it suffices to show that $\jmath_A^{\dashv}[{\rm Spec}^{\infty}\,(A)] = \{ 1\}.$\\
		
		%(see \textbf{Lemma 4.8} of \cite{BellSlomson})
		In order to prove it, we need the following:\\
		
		\hspace{0.5cm}\textbf{Claim 2.1:} $(\jmath_A(e) = {\rm Spec}^{\infty}\,(A)) \iff (e \in A^{\times})$.\\
		
		In fact, $a \in A^{\times} \Rightarrow D^{\infty}(a) = {\rm Spec}^{\infty}\,(A)$. On the other hand, if $a \notin A^{\times}$ then there is some maximal $\mathcal{C}^{\infty}-$radical prime ideal $\mathfrak{m}\in {\rm Spec}^{\infty}\,(A)$ such that $a \in \mathfrak{m}$. If $\mathfrak{m}\notin D^{\infty}(a)$ then $\jmath_A(a) = D^{\infty}(a) \neq {\rm Spec}^{\infty}\,(A)$. Hence $\jmath_A(a) = {\rm Spec}^{\infty}\,(A) \Rightarrow a \in A^{\times}$, and the claim is proved. \\
		
		Let $e \in \widetilde{B}(A)$ be such that $\jmath_A(e) = D^{\infty}(e) = {\rm Spec}^{\infty}\,(A)$, so by \textbf{Claim 2.1}, it follows that $e \in A^{\times}$ and since $e$ is idempotent, $e = 1$, that is $\jmath_A^{\dashv}[\operatorname{Spec}^{\infty}\,(A)] = \{1\}$. \\
		
		The injective map $\jmath_A : \widetilde{B}(A) \to {\rm Spec}^{\infty}\,(A)$ suggests that the idempotent elements of the  Boolean algebra $\widetilde{B}(A)$ associated with a $\mathcal{C}^{\infty}-$ring $A$ hold a strong relationship with the canonical basis of the Zariski topology of ${\rm Spec}^{\infty}(A)$.  We are going to show that these idempotent elements, in the case of the von Neumann-regular $\mathcal{C}^{\infty}-$rings, represent \textit{all the Boolean algebras}.

		\begin{theorem}\label{jordao}Let $A$ be a von Neumann regular  $\mathcal{C}^{\infty}-$ring. The map:
			$$\begin{array}{cccc}
				\jmath_A: & \widetilde{B}(A) & \rightarrow & {\rm Clopen}\,({\rm Spec}^{\infty}(A)) \\
				& e & \mapsto & D^{\infty}(e) = \{ \mathfrak{p} \in {\rm Spec}^{\infty}\,(A)\mid e \notin \mathfrak{p}\}
			\end{array}$$
			is an isomorphism of Boolean algebras.
		\end{theorem}
		\begin{proof}
			We already know that $\jmath_A$ is an injective Boolean algebras homomorphism, so it suffices to prove that it is also surjective. This is achieved noting that given any $a_1, a_2, \cdots, a_n \in A$, there is some $b \in A$ such that $D^{\infty}(a_1)\cup D^{\infty}(a_2)\cup \cdots \cup D^{\infty}(a_n) = D^{\infty}(b)$, which is proved by induction, using item (iii) of \textbf{Lemma 1.4} of \cite{rings2} for the case $n=2$.\\
			
			Since $A$ is a von Neumann-ring, given this $b\in A$, there is an idempotent element, $e$, such that $(b) = (e)$, so $D^{\infty}(b)=D^{\infty}(e)$. Thus, $\jmath_A$ is surjective, as claimed.
		\end{proof}
		
		\begin{theorem}\label{uhum}We have the following diagram of categories, functors and a natural isomorphism:
			\begin{equation}
				\label{diagrama}
				\xymatrixcolsep{3pc}\xymatrix{
					\mathcal{C}^{\infty}{\rm \bf vNRing} \ar@/_/[ddrr]_{\widetilde{B}} \ar[rr]^{{\rm Spec}^{\infty}} & & {\rm \bf BoolSp} \ar[dd]^{{\rm Clopen}}\\
					& \ar@2[ur]_{\jmath} & \\
					& & {\rm \bf BA}
				}
			\end{equation}
		\end{theorem}
		\begin{proof}
			First note that since $A$ is a von Neumann-regular $\mathcal{C}^{\infty}-$ring, the set of the compact open subsets of (the Boolean space) ${\rm Spec}^{\infty}\,(A)$ equals ${\rm Clopen}\,({\rm Spec}^{\infty}\,(A))$.
			
			On the one hand, given a von Neumann regular $\mathcal{C}^{\infty}-$ring $A$, we have ${\rm Clopen}\,({\rm Spec}^{\infty}\,(A)) = \jmath_A[\widetilde{B}\,(A)] = \{ D^{\infty}(e) \mid e \in \widetilde{B}(A)\}$.
			
			For every von Neumann regular $\mathcal{C}^{\infty}-$ring $A$, by \textbf{Theorem \ref{jordao}}, we have the following isomorphism of Boolean algebras:
			
			$$\begin{array}{cccc}
				\jmath_A: & \widetilde{B}(A) & \rightarrow & {\rm Clopen}\,({\rm Spec}^{\infty}(A)) \\
				& e & \mapsto & D^{\infty}(e)
			\end{array}$$
			
			It is easy to see that for every $\mathcal{C}^{\infty}-$homomorphism $f: A \to A'$, we have the following commutative diagram:
			
			$$\xymatrixcolsep{3pc}\xymatrix{
				\widetilde{B}(A) \ar[d]_{\widetilde{B}(f)} \ar[r]^(0.37){\jmath_A} & {\rm Clopen}\,({\rm Spec}^{\infty}\,(A)) \ar[d]^{{\rm Clopen}\,({\rm Spec}^{\infty}(f))}\\
				\widetilde{B}(A') \ar[r]_(0.37){\jmath_{A'}} & {\rm Clopen}\,({\rm Spec}^{\infty}\,(A'))}$$
			
			In fact, given $e\in \widetilde{B}(A)$, we have, on the one hand, $\jmath_A(e)=D_{A}^{\infty}(e)$ and ${\rm Clopen}({\rm Spec}^{\infty}(f))(D_A^{\infty}(e)) = D_{A'}^{\infty}(f(e))$. On the other hand, $\jmath_{A'}\circ \widetilde{B}(f)(e) = \jmath_{A'}(f(e))= D_{A'}^{\infty}(f(e))$, so the diagram \eqref{diagrama} commutes.\\
			
			Thus, $\jmath$ is a natural transformation and $\xymatrix{\jmath: \widetilde{B} \ar@2[r] & {\rm Clopen}\circ {\rm Spec}^{\infty}}$ is a natural isomorphism and the diagram ``commutes'' (up to natural isomorphism).
		\end{proof}
		
		The following lemma is a well-known result, of which the authors could not find a proof anywhere in the current literature. The authors provide a proof in \textbf{Lemma 6} of \cite{berni2019neumann}.
		
		\begin{lemma}\label{niedia}Let $(X,\tau)$ be a Boolean topological space, and let:
			\begin{multline*}
				\mathcal{R} = \{ R \subseteq X \times X \mid (R \,\, \text{is an equivalence relation on} \, X)\& \\ \& (({X}/{R}) \, \text{is a discrete compact space}) \}
			\end{multline*}
			which is partially ordered by inclusion. Whenever $R_i, R_j \in \mathcal{R}$ are such that $R_j \subseteq R_i$, we have the continuous surjective map:
			$$\begin{array}{cccc}
				\mu_{R_jR_i}: & ({X}/{R_j}) & \twoheadrightarrow & ({X}/{R_i}) \\
				& [x]_{R_j} & \mapsto & [x]_{R_j}
			\end{array}$$
			so we have the inverse system $\{ ({X}/{R_i}) ; \mu_{R_jR_i}: ({X}/{R_j}) \to ({X}/{R_i}) \}$. By definition (see, for instance, \cite{james1966topology}),

			\begin{multline*}
				\varprojlim_{R \in \mathcal{R}} ({X}/{R}) = \lbrace ([x]_{R_i})_{R_i \in \mathcal{R}} \in \prod_{R \in \mathcal{R}} ({X}/{R}) \mid R_j \subseteq R_i \to ([x]_{R_i} = \mu_{R_jR_i}([x]_{R_j}) \rbrace
			\end{multline*}
			
			Let $X_{\infty}$ denote $\varprojlim_{R \in \mathcal{R}} \frac{X}{R}$, so we have the following cone:
			$$\xymatrix{
				& X_{\infty} \ar[dl]_{\mu_{R_j}} \ar[dr]^{\mu_{R_i}} & \\
				({X}/{R_j}) \ar[rr]^{\mu_{R_jR_i}} & & ({X}/{R_i})
			}$$
			
			We consider $X_{\infty}$ together with the induced subspace topology of $\prod_{R \in \mathcal{R}} ({X}/{R})$.\\
			
			By the universal property of $X_{\infty}$, there is a unique continuous map map $\delta: X \to X_{\infty}$ such that the following diagram commutes:
			
			$$\xymatrixcolsep{3pc}\xymatrix{
				& X \ar@/_2pc/[ddl]_{q_{R_j}} \ar@/^2pc/[ddr]^{q_{R_i}} \ar@{-->}[d]^{\exists ! \delta}& \\
				& X_{\infty} \ar[dl]_{\mu_{R_j}} \ar[dr]^{\mu_{R_i}} & \\
				({X}/{R_j}) \ar[rr]_{\mu_{R_jR_i}} & & ({X}/{R_i})
			}$$
			
			We claim that such a $\delta: X \to X_{\infty}$ is a homeomorphism, so:
			
			$$X \cong \varprojlim_{R \in \mathcal{R}} ({X}/{R})$$
			that is, $X$ a profinite space.
		\end{lemma}
		\begin{proof}
			See \textbf{Lemma 6} of \cite{berni2019neumann}.
		\end{proof}

		Let ${\rm \bf BoolSp}$ be the category whose objects are all the Boolean spaces and whose morphisms are all the continuous functions between Boolean spaces. Given any Boolean space $(X,\tau)$, let
		\begin{multline*}\mathcal{R}_X = \{ R \subseteq X \times X \mid R\,\, \mbox{is \,\,\, an \,\,\, equivalence \,\,\, relation\,\,\, on}\,\, X \,\, \mbox{and} \\
			({X}/{R})\,\,\, \mbox{is \,\,\, discrete \,\,\, and \,\,\, compact}\}
		\end{multline*}
		
		We are going to describe an equivalence functor between \textbf{BoolSp} and the category of profinite topological spaces: this is a known result, but we cannot found a reference containing a detailed description. \\

		First we note that given any continuous function $f: X \to X'$ and any $R' \in \mathcal{R}_{X'}$,
		
		$$R_{f}:= (f \times f)^{\dashv}[R'] \subseteq X \times X$$
		
		\noindent is an equivalence relation on $X$ and the following diagram:
		
		$$\xymatrixcolsep{3pc}\xymatrix{
			X \ar[r]^{f} \ar[d]_{p_{R_f}} & X' \ar[d]^{p_R'}\\
			({X}/{R_f}) \ar[r]^{f_{R_fR'}} & ({X'}/{R'})}$$
		
		\noindent where $p_{R_f}: X \to ({X}/{R_f})$ and $p_{R'}: X' \to ({X'}/{R'})$ are the canonical projections, commutes.
		
		We know, by \textbf{Theorem 4.3} of \cite{james1966topology} that $f_{R_fR'}: ({X}/{R_f}) \to ({X'}/{R'})$ is a continuous map, and it is easy to see that $f_{R_fR'}$ is injective, as we are going to show.
		
		Given any $[x]_{R_f},[y]_{R_f} \in \dfrac{X}{R_f}$ such that $[x]_{R_f} \neq [y]_{R_f}$, \textit{i.e.}, such that $(x,y) \notin R_f$, we have $(f(x),f(y)) \notin R'$, \textit{i.e.}, $[f(x)]_{R'} \neq [f(y)]_{R'}$. Thus, since: $$f_{R_fR'}([x]_{R_f}) = (f_{R_fR'}\circ p_{R_f})(x) = (p_{R'} \circ f)(x) = [f(x)]_{R'}$$
		and
		$$f_{R_fR'}([y]_{R_f}) = (f_{R_fR'}\circ p_{R_f})(y) = (p_{R'} \circ f)(y) = [f(y)]_{R'}$$
		it follows that $f_{R_fR'}([x]_{R_f}) \neq f_{R_fR'}([y]_{R_f})$.
		
		Since $f_{R_fR'}: \frac{X}{R_f} \to \frac{X'}{R'}$ is an injective continuous map and $\frac{X'}{R'}$ is discrete, it follows that given any $[x']_{R'} \in \frac{X'}{R'}$:
		$$f_{R_fR'}^{\dashv}[\{ [x']_{R'}\}] = \begin{cases}
			\varnothing, \,\, \mbox{if}\,\, [x']_{R'} \notin f_{R_fR'}[X/R_f],\\
			\{ *\}, \,\, \mbox{otherwise}
		\end{cases},$$
		so every singleton of $\frac{X}{R_f}$ is an open subset of $\frac{X}{R_f}$, and $\frac{X}{R_f}$ is discrete. Also, since $X$ is compact, $\frac{X}{R_f}$ is compact, and it follows that and $R_f \in \mathcal{R}_X$.\\
		
		Now, if $R_1',R_2' \in \mathcal{R}_{X'}$ are such that $R_1' \subseteq R_2'$, then ${R_1'}_f \subseteq {R_2'}_f$. In fact, given $(x,y) \in {R_1'}_f$, we have $(f \times f)(x,y) \in R_1'$, and since $R_1' \subseteq R_2'$, it follows that $(f \times f)(x,y) \in R_2'$, so $(x,y) \in {R_2'}_f$.\\
		
		Let $X,Y,Z$ be Boolean spaces, $X \stackrel{f}{\rightarrow} Y$ and $Y \stackrel{g}{\rightarrow} Z$ be two Boolean spaces homomorphisms and $R \in \mathcal{R}_Z$. We easily see that $((g\circ f) \times (g \circ f))^{\dashv}[R] = (f \times f)^{\dashv}[(g \times g)^{\dashv}[R]].$\\
		
		%In fact,
		%\begin{multline*}(x,x') \in ((g\circ f) \times (g \circ f))^{\dashv}[R] \iff (g(f(x)),g(f(x'))) \in R \iff \\
		%\iff (f(x),f(x')) \in (g \times g)^{\dashv}[R] \iff (x,x') \in (f\times f)^{\dashv}[(g \times g)^{\dashv}[R]].
		%\end{multline*}
		
		Denoting $T:= (f \times f)^{\dashv}[(g \times g)^{\dashv}[R]]$ and $S:=(g \times g)^{\dashv}[R]$, we have the following commutative diagram:

		$$\xymatrixcolsep{3pc}\xymatrix{
			X \ar[r]^{f} \ar[d]_{\mu_T} & Y \ar[r]^{g} \ar[d]^{\mu_S} & Z \ar[d]^{\mu_R}\\
			({X}/{T}) \ar@/_3pc/[rr]_{(g\circ f)_{TR}}\ar[r]^{f_{TS}} & ({Y}/{S}) \ar[r]^{g_{SR}} & ({Z}/{R})
		}$$
		
		Given a continuous map between Boolean spaces, $f: X \to X'$, we can define a map $\check{f}: \varprojlim_{R \in \mathcal{R}_X}({X}/{R}) \to \varprojlim_{R' \in \mathcal{R}_{X'}} ({X'}/{R'})$ in a functorial manner.\\
		
		Let $R',S' \in \mathcal{R}_{X'}$ be any two equivalence relations such that  $R' \subseteq S'$, so given $f: X \to X'$ the following diagram commutes:
		
		$$\xymatrixcolsep{3pc}\xymatrix{
			({X}/{S_f'}) \ar[r]^{\mu_{R_f'S_f'}} \ar[d]_{f_{S_f'S'}} & ({X'}/{R_f'}) \ar[d]^{f_{R_f'R'}}\\
			({X'}/{S'}) \ar[r]^{\mu_{R'S'}} & ({X'}/{R'})
		}$$
		
		\noindent and since the diagram \eqref{diag3} commutes, the diagram \eqref{diag4} also commutes.

		\begin{equation}\label{diag3}
			\xymatrixcolsep{3pc}\xymatrix{
				& X_{\infty} \ar@/_/[dl]_{\mu_{S_f'}} \ar@/^/[dr]^{\mu_{R_f'}}& \\
				({X}/{S_f'}) \ar[rr]_{\mu_{R_f'S_f'}} & & ({X}/{R_f'})}
		\end{equation} \\ \begin{equation}\label{diag4}
			\xymatrixcolsep{3pc}\xymatrix{
				& X_{\infty} \ar@/_/[dl]_{f_{S_f'S'}\circ \mu_{S_f'}} \ar@/^/[dr]^{f_{R_f'R'}\mu_{R_f'}}& \\
				({X'}/{S'}) \ar[rr]_{\mu_{R'S'}} & & ({X'}/{R'})
		}\end{equation}

		By the universal property of ${X'}_{\infty}$, there is a unique $\check{f}: X_{\infty} \to {X'}_{\infty}$ such that the following prism is commutative:
		
		$$\xymatrix @!0 @R=4pc @C=6pc {
			X_{\infty} \ar[rr]^{\mu_{R_f'}} \ar[rd]^{\mu_{S_f'}} \ar@{.>}[dd]_{\exists ! \check{f}} && \frac{X}{R_f'} \ar[dd]^{f_{R_f'R'}} \\
			& \frac{X}{S_f'} \ar[ru]_{\mu_{R_f'S_f'}} \ar[dd]^(.3){f_{S_f'S'}} \\
			{{X'}_{\infty}} \ar[rr] |!{[ur];[dr]}\hole \ar[rd]^{\mu_{S'}'}^(.2){\mu_{R'}} && \frac{X'}{R'} \\
			& \frac{X}{S'} \ar[ru]_{\mu_{R'S'}} }$$

		It can be proven  that:
		
		$$\begin{array}{cccc}
			\Delta : & {\rm \bf BoolSp} & \rightarrow & {\rm\bf ProfinSp} \\
			& X & \mapsto & X_{\infty} = \varprojlim_{R \in \mathcal{R}_X} ({X}/{R})\\
			& \xymatrix{X \ar[r]^{f}&X'} & \mapsto & \xymatrix{X_{\infty} \ar[r]^{\check{f}}& {X'}_{\infty}}
		\end{array}$$
		
		\noindent is a functor using the uniqueness properties regarding $X_{\infty}$ and $\check{f}$ (for details, see \cite{berni2019neumann}).

		\begin{theorem} \label{boospvn} Let $\mathbb{K}$ be a $\mathcal{C}^{\infty}-$field.
			Following the notation of \textbf{Lemma \ref{niedia}}, define the
			contravariant functor:
			
			$$\begin{array}{cccc}
				\widehat{k}: & {\rm \bf BoolSp} & \rightarrow & \mathcal{C}^{\infty}{\rm \bf vNRing} \\
				& (X,\tau) & \mapsto & A_X:= \varinjlim_{R \in \mathcal{R}}
				\mathbb{K}^{U\left( \frac{X}{R}\right)}
			\end{array}$$
			
			Then there is a natural isomorphism:
			$$\epsilon : \operatorname{Id}_{\rm \bf BoolSp} \overset{\cong}{\Rightarrow} {\rm Spec}^{\infty}
			\circ \widehat{k}$$

			Therefore:
			\begin{itemize}
				%\item{ $\widecheck{K}=\textcolor{red}{\widehat{k}\circ \operatorname{Stone}}: \text{\bf BA} \rightarrow \mathcal{C}^{\infty}-\text{\bf vNRing}$}  
				%\item{$\theta: \operatorname{Id}_{\text{\bf BA}} \Longrightarrow \widetilde{B}\circ \textcolor{red}{\widehat{k}\circ \operatorname{Stone}} = \widetilde{B}\circ \textcolor{red}{\widecheck{K}}$ natural isomorphism}
				\item{The functor $\widehat{k}$ is faithful;}
				\item{The functor ${\rm Spec}^{\infty}: \mathcal{C}^{\infty}{\rm \bf vNRing} \rightarrow {\rm \bf BoolSp}$  is ``full up to conjugation'';}
				\item{The functor ${\rm Spec}^{\infty}: \mathcal{C}^{\infty}{\rm \bf vNRing} \rightarrow {\rm \bf BoolSp}$  is isomorphism-dense.  In particular: for each $(X,\tau)$ be a Boolean space, there is a von Neumann-regular $\mathcal{C}^{\infty}-$ring, $A_X$, such that ${\rm Spec}^{\infty}\,(A_X) \approx X.$}
			\end{itemize}
		\end{theorem}
		\begin{proof}
			%By \textbf{Lemma \ref{tadeu}},
			%$${\rm Spec}^{\infty}\,(R_X) = \varprojlim_{R \in \mathcal{R}} {\rm
				%Spec}^{\infty}\, (\mathbb{K}^{U\left( \frac{X}{R}\right)}).$$

			By the \textbf{Theorem 34}, p. 118 of \cite{berni2019topics},
			
			$${\rm Spec}^{\infty}(A_X) \approx \varprojlim_{R \in \mathcal{R}}
			{\rm Spec}^{\infty}\,(\mathbb{K}^{U\left({X}/{R}\right)}).$$
			
			By the \textbf{Theorem 33}, p. 118 of \cite{berni2019topics},
			
			$${\rm Spec}^{\infty}\,(\mathbb{K}^{U\left({X}/{R}\right)})
			\approx  {X}/{R},$$
			so
			
			$$\varprojlim_{R \in \mathcal{R}} {\rm
				Spec}^{\infty}\,(\mathbb{K}^{U\left(\frac{X}{R}\right)}) \approx
			\varprojlim_{R \in \mathcal{R}} ({X}/{R})
			\approx X$$
			
			Since the homeomorphisms above are natural, just take $\epsilon_X : X
			\to {\rm Spec}^\infty(A_X)$ as the composition of these homeomorphisms. 
			
			In particular, ${\rm Spec}^{\infty}\,(A_X) \approx X$ and ${\rm Spec}^{\infty}$ is an isomorphism-dense functor.
			
			Let $\phi : X \to X'$ be a continuous function between Boolean spaces. Since  $\epsilon$ is a natural isomorphism, we have $\phi = \epsilon_{X'}^{-1} \circ {\rm Spec}^{\infty}(\widehat{k}(\phi)) \circ \epsilon_X$. In particular, there exists a homomorphism of  von
			Neumann regular $\mathcal{C}^{\infty}-$rings $f : A' \to A$ and homeomorphisms of Boolean spaces $\psi : X \to {\rm Spec}^{\infty}(A)$ and $\psi' : X' \to {\rm Spec}^{\infty}(A')$, such that
			$$ \phi = {\psi'}^{-1} \circ  {\rm Spec}^{\infty}(f) \circ \psi,$$
			thus ${\rm Spec}^{\infty}$ is a full up to conjugatization functor.
			
			Let $\phi, \psi : X \to X'$ be continuous functions between Boolean spaces such that $\widehat{k}(\phi) = \widehat{k}(\psi)$. Then 
			$$\phi = \epsilon_{X'}^{-1} \circ {\rm Spec}^{\infty}(\widehat{k}(\phi)) \circ \epsilon_X = \epsilon_{X'}^{-1} \circ {\rm Spec}^{\infty}(\widehat{k}(\psi)) \circ  \epsilon_X = \psi,$$
			thus $\widehat{k}$ is a faithful functor.
			
		\end{proof}

		\begin{theorem}\label{qmq} Let $\mathbb{K}$ be a $\mathcal{C}^{\infty}-$field.
			Defining the covariant functor (composition of contravariant functors):
			
			$$\check{K} = \widehat{k}\circ {\rm Stone} : {\rm \bf BA} \to
			\mathcal{C}^{\infty}{\rm \bf vNRing}.$$
			
			Then there is a natural isomorphism
			$$\theta : {\rm Id}_{\rm \bf BA} \overset{\cong}\Rightarrow \widetilde{B}
			\circ \check{K}.$$

			Therefore:
			%, the functor $\widetilde{B} : \mathcal{C}^{\infty}{\rm \bf vNRing} \rightarrow {\rm \bf BA}$ is full and \emph{isomorphism-dense}. 
			\begin{itemize}
				%\item{ $\widecheck{K}=\textcolor{red}{\widehat{k}\circ \operatorname{Stone}}: \text{\bf BA} \rightarrow \mathcal{C}^{\infty}-\text{\bf vNRing}$}  
				%\item{$\theta: \operatorname{Id}_{\text{\bf BA}} \Longrightarrow \widetilde{B}\circ \textcolor{red}{\widehat{k}\circ \operatorname{Stone}} = \widetilde{B}\circ \textcolor{red}{\widecheck{K}}$ natural isomorphism}
				\item{The functor $\check{K}$ is faithful;}
				\item{The functor $ \widetilde{B}$ is full up to conjugation;}
				\item{The functor $ \widetilde{B}: \mathcal{C}^{\infty}-\text{\bf vNRing} \rightarrow \text{\bf BA}$ is isomorphism-dense. In particular: given any $\mathcal{C}^{\infty}-$field $\mathbb{K}$ and any Boolean algebra $B$, there is a von
					Neumann regular $\mathcal{C}^{\infty}-$ring which is a $\mathbb{K}-$algebra,
					$\check{K}(B)$, such that $\widetilde{B}(\check{K}(B)) \cong B$.}
			\end{itemize}
		\end{theorem}
		
		\begin{proof} This follows directly by a combination of
			the \textbf{Theorem \ref{boospvn}} above,
			Stone duality (\textbf{Remark \ref{stone-dual}}), \textbf{Theorem
				\ref{uhum}} and \textbf{Theorem \ref{jordao}}.
		\end{proof}
		
		The diagram below summarizes the main functorial connections
		established in this section:

		$$\xymatrixcolsep{3pc}
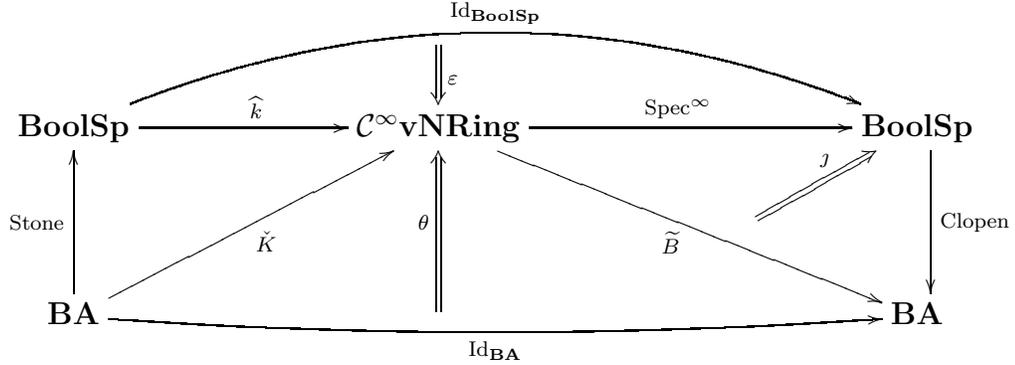
\xymatrix{
			& & \ar@2[d]^{\varepsilon}& & & & \\
			{\rm \bf BoolSp}\ar@/^3pc/[rrrrr]^{{\rm Id}_{\rm \bf BoolSp}}
			\ar[rr]^{\widehat{k}} & & \mathcal{C}^{\infty}{\rm \bf vNRing}
			\ar[rrr]^{{\rm Spec}^{\infty}}
			\ar[ddrrr]_{\widetilde{B}} & & & {\rm \bf BoolSp} \ar@<1ex>[dd]^{{\rm
					Clopen}}\\
			& & & & \ar@2[ur]^{\jmath} & \\
			{\rm \bf BA} \ar[uu]^{{\rm Stone}} \ar[uurr]_{\check{K}}
			\ar@/_/[rrrrr]_{{\rm Id}_{{\rm \bf BA}}} &  & \ar@2[uu]^{\theta} & & & {\rm \bf BA}}$$

		%\begin{remark}
		%For more information on the lay-out of this article for the {\em Cahiers}, the official {\em Guide to Authors} can be found at the following address: 
		%\begin{center}
		%{\tt http://ehres.pagesperso-orange.fr/Cahiers/Ctgdc.htm}.
		%\end{center}
		%\end{remark}
		
		%{\color{red}{
				
				We finish this work with the following:
				
				%\begin{frame}{}
				%\begin{tcolorbox}
				%\begin{itemize}
				%    \item{ $\widecheck{K}=\textcolor{red}{\widehat{k}\circ \operatorname{Stone}}: \text{\bf BA} \rightarrow \ci\text{\bf vNRing}$} \pause 
				%    \item{$\theta: \operatorname{Id}_{\text{\bf BA}} \Longrightarrow \widetilde{B}\circ \textcolor{red}{\widehat{k}\circ \operatorname{Stone}} = \widetilde{B}\circ \textcolor{red}{\widecheck{K}}$ natural isomorphism}\pause
				%    \item{$\therefore \widecheck{K}$ is faithful; }\pause
				%    \item{$\therefore \widetilde{B}: \ci\text{\bf vNRing} \rightarrow \text{\bf BA}$ is isomorphism-dense;}\pause 
				%    \item{$\therefore \widetilde{B}$ is full up to conjugation;}
				%\end{itemize}
				
				%\end{tcolorbox}
				%\end{frame}

				%\begin{frame}{}
				%\begin{tcolorbox}
				%\begin{center}
				%   \begin{figure}
					%       \centering
					%       \includegraphics[scale=0.25]{diagramaGrandao.png}
					%   \end{figure}
				%\end{center}
				%\end{tcolorbox}
				%\end{frame}
				
				%\begin{frame}
				\begin{remark}{ The theorem above leads us to the natural question(s): Is $\widetilde{B}$ an equivalence of categories, possibly with with $\check{K}$  being the quasi-inverse of $\widetilde{B}$, for some $\mathcal{C}^{\infty}-$field $\mathbb{K}$?}
					
					%\begin{tcolorbox}
					\begin{itemize}
						\item{For every $\mathcal{C}^{\infty}-$field $\mathbb{K}$ with $\operatorname{card}\,(\mathbb{K})> \operatorname{card}\,(\mathbb{R})$, $\check{K}$ {\bf can not be} a quasi-inverse of $\widetilde{B}$; in fact, since for every Boolean algebra $A$, $\check{K}(A)$ is in particular a $\mathbb{K}$-algebra, we have $\operatorname{card}(\check{K}(A)) \geq \operatorname{card}(\mathbb{K})$, whenever $A \ncong \{0\}$. Let $V$ be a von Neumann-regular $\mathcal{C}^{\infty}-$ring; since the  class of  $\mathcal{C}^{\infty}-$fields  is first-order axiomatizable in the language of $\mathcal{C}^{\infty}-$rings, and every $\mathcal{C}^{\infty}-$field is an infinite set (since it is, in particular, a non-trivial $\mathbb{R}$-algebra), then by the \textbf{L\"{o}wenheim-Skolem Theorem} (upward, which we denote by $\uparrow\operatorname{LS}$), there is a $\mathcal{C}^{\infty}-$field $\mathbb{K}$ such that $\operatorname{card}\,(\mathbb{K}) {>} \operatorname{card}\,(V)$. Therefore,  $\operatorname{card}\,(\check{K}(\widetilde{B}\,(V))) \geq \operatorname{card}\,(\mathbb{K}) \stackrel{\uparrow\operatorname{LS}}{>} \operatorname{card}\,(V)$ and, thus, $\check{K}(\widetilde{B}\,(V)) \not\cong V$. Since there exist non-trivial von Neumann-regular $\mathcal{C}^{\infty}-$rings $V$ with $\operatorname{card}\,(V) =  \operatorname{card}\,(\mathbb{R})$, this shows that there is no $\mathcal{C}^{\infty}-$field ,$\mathbb{K}$, with $\operatorname{card}\,(\mathbb{K})> \operatorname{card}\,(\mathbb{R})$, such that $\check{K}$ is a quasi-inverse of $\widetilde{B}$.}
						\item{It is important to stress that the functor $\widetilde{B}$ \textbf{is not} an equivalence of categories. {In fact, by $\uparrow\operatorname{LS}$, there are $\mathcal{C}^{\infty}-$fields, $\mathbb{K}$, with $\operatorname{card}\,(\mathbb{K})> \operatorname{card}\,(\mathbb{R})$. We saw above that $\widetilde{B}$  is a ``full up to conjugation functor'', but if $\widetilde{B}$ were a full functor,  then the Boolean algebra isomorphism $\widetilde{B}\,(\mathbb{K}) \cong {\bf 2} \stackrel{\operatorname{id}_{\bf 2}}{\to}   {\bf 2} \cong \widetilde{B}\,(\mathbb{R})$ should be the image of some $\mathcal{C}^{\infty}-$homomorphism $\mathbb{K} \rightarrow \mathbb{R}$: this is a contradiction since  a $\mathcal{C}^{\infty}-$homomorphism between $\mathcal{C}^{\infty}-$fields must be injective. Therefore $\widetilde{B}$ is not a full functor, thus $\widetilde{B}$ is  not an equivalence of categories. }}
						%    \item{We suspect that $\widetilde{B}$ is not a faithful functor (combinatorial reasoning): SE NAO COLOCAR ARGUMENTOS TEMOS QUE OMITIR ISTO...}
					\end{itemize}
				\end{remark}

				%%%%%%%%%%%%%%%%%%%%%%%%%%%%%%%%%%%%%%%%%%%%%%%%%%%%%%%%%%%%%%%%%%%%%%%%%%%%%%%%%%%%%%%%%%%%%%%%%%%%%%%%%%%%%%%%%%%%%%%%%%%%%%%%%%%%%%%%%%%%%%%%%%

			\end{document}